\documentclass[reqno,11pt]{amsart}

\usepackage{amsmath,amssymb,color,amsthm}
\usepackage{amsfonts, amscd, epsfig, amsmath, amssymb,enumerate}
\usepackage{graphicx}
\usepackage{graphics}
\usepackage{xcolor}
\usepackage{mathrsfs}
\usepackage{todonotes}
\usepackage{soul}
 \usepackage[usenames,dvipsnames]{pstricks}
 \usepackage{pstricks-add}
 \usepackage{epsfig}
 \usepackage{pst-grad} % For gradients
 \usepackage{pst-plot} % For axes
 \usepackage[space]{grffile} % For spaces in paths
 \usepackage{etoolbox} % For spaces in paths
\usepackage[margin=3cm]{geometry}
 \makeatletter % For spaces in paths
 \patchcmd\Gread@eps{\@inputcheck#1 }{\@inputcheck"#1"\relax}{}{}
 \makeatother

\newtheorem{theorem}{Theorem}[section]
\newtheorem{lemma}[theorem]{Lemma}
\newtheorem{proposition}[theorem]{Proposition}
\newtheorem{corollary}[theorem]{Corollary}
\newtheorem{remark}[theorem]{Remark}
\newtheorem{definition}[theorem]{Definition}
\newtheorem*{example*}{Example}

\usepackage{tikz}
\usetikzlibrary{backgrounds}
\usetikzlibrary{patterns,fadings}
\usetikzlibrary{arrows,decorations.pathmorphing}
\usetikzlibrary{decorations}
\usetikzlibrary{calc}
\usetikzlibrary{shapes.misc}

\def \CS {Cauchy--Schwarz }

\usepackage{setspace}
\definecolor{light-gray}{gray}{0.95}
\usepackage{float}
\usepackage[colorlinks=true,linkcolor=blue,citecolor=magenta]{hyperref}
\def\centerarc[#1](#2)(#3:#4:#5){\draw[#1] ($(#2)+({#5*cos(#3)},{#5*sin(#3)})$) arc (#3:#4:#5);}

\numberwithin{equation}{section}
\numberwithin{figure}{section}

\newcommand{\<}{\big\langle}
\renewcommand{\>}{\big\rangle}

\newcommand{\R}{\mathbb R}

\newcommand{\N}{\mathbb N}
\renewcommand{\P}{\mathbb P}

\newcommand{\E}{\mathbb E}
\newcommand{\Q}{\mathbb Q}

\newcommand{\Acal}{\mathcal{A}} 
 
\newcommand{\Mcal}{\mathcal{M}}

\newcommand{\Ecal}{\mathcal{E}}

\allowdisplaybreaks %%Pour \UTF{00E9}viter les grands espacements verticaux.

\title[Asymmetric simple exclusion with source]{Hydrodynamics for asymmetric simple exclusion on a finite segment with Glauber-type source}

\author{Lu \textsc{Xu}}
\address{Gran Sasso Science Institute, 67100 L'Aquila AQ, Italy.}
\email{lu.xu@gssi.it}

\author{Linjie \textsc{Zhao}}
\address{School of Mathematics and Statistics, Huazhong University of Science \& Technology, Wuhan 430074, China.}
\address{Hubei Key Laboratory of Engineering Modeling and Scientific Computing, Huazhong University of Science and Technology, Wuhan 430074, China}
\email{linjie\_zhao@hust.edu.cn}


\begin{document}

\begin{abstract}
We consider an open interacting particle system on a finite lattice.
The particles perform asymmetric simple exclusion and are randomly created or destroyed at all sites, with rates that grow rapidly near the boundaries.
We study the hydrodynamic limit for the particle density at the hyperbolic space-time scale and obtain the entropy solution to a boundary-driven quasilinear conservation law with a source term. Different from the usual boundary conditions introduced in \cite{BLN79,Otto96}, discontinuity (boundary layer) does not formulate at the boundaries due to the strong relaxation scheme.
\end{abstract}

\maketitle

\section{Introduction}
\label{sec:intr}

%Exclusion process is one of the most important models in the theory of interacting particle systems and statistical physics, see \cite{liggett99}. In the dynamics, particles perform random walks on some graph subjected to the exclusion rule, that is, there is at most one particle at each site of the underlying graph.  We are interested in the large-scale behaviors of the process, which is usually done by a space-time scaling and is called hydrodynamic limits in the literature, see \cite{klscaling}.
%
In the past decades, hydrodynamic limit for interacting particle system with boundary effect has attracted a lot of attention \cite{Baha12, baldasso2017exclusion, BCGS23, bernardin20192low, BGJ21, DPTV11, ESZ23, landim2007stationary, xu2022hydrodynamics}.
The limit captures the evolution of the conserved field of the microscopic dynamics as \emph{hydrodynamic equation}, at the macroscopic time scale that the dynamics is equilibrated locally.
Most of these works focus on symmetric dynamics and the Fick's law of diffusion or fractional diffusion with various types of boundary conditions, see, e.g., \cite{baldasso2017exclusion, BCGS23, bernardin20192low, BGJ21, DPTV11, landim2007stationary}.

%In \cite{xu2022hydrodynamics,xu2021hydrodynamic}, the first author considered the boundary-driven asymmetric exclusion process on  the line segment $\{0,1,2,\ldots,N\}$.  In the bulk, a particle jumps to its right neighbor at rate $1/2 < p \leq 1$, and to its left at rate $1-p$.  On top of that, particles can also be created and destroyed at the boundaries $\{0,N\}$. Under suitable conditions on the accelerate rates of the dynamics, the first author showed that the hydrodynamic equation is given by the hyperbolic conservation law with boundary conditions.
%
For asymmetric mass-conserving systems, the dynamics reaches local equilibrium at the hyperbolic time scale, and the hydrodynamic equations are given by hyperbolic transport equations \cite{rezakhanlou91}.
When nonlinear interaction exists, these equations are featured by discontinuous phenomenon both inside the domain (shock wave) and at the boundary (boundary layer).
The non-regularity becomes the main obstacle in deducing the hydrodynamic limit.
Asymmetric simple exclusion process (ASEP) with open boundaries is the simplest model.
In its dynamics, each particle performs an asymmetric random walk on the finite lattice $\{1,\ldots,N-1\}$ under the exclusion rule: two particles cannot occupy the same site simultaneously.
Particles are created and annihilated randomly at sites $1$ and $N-1$, modeling the exchange of mass between the system and two external reservoirs at given densities.
In \cite{Baha12, xu2022hydrodynamics, xu2021hydrodynamic}, the hydrodynamic limit for the mass density of open ASEP is proved to be Burgers equation with boundary conditions introduced in \cite{BLN79, Otto96}.
Due to the discontinuous nature, these boundary conditions do not prescribe the density at boundary, even when the reservoir dynamics is overwhelmingly accelerated compared to the exclusion \cite{xu2021hydrodynamic}.
Instead, they impose a set of possible values for the boundary density.
The hydrostatic limit for the same dynamics is studied in \cite{Baha12, DEHP93}: the stationary density profile is the stationary solution to the hydrodynamic equation.
It is determined by the boundary data through a variational property \cite{PopkovS99}: the stationary flux is maximized if the density gradient is opposite to the drift, and is minimized otherwise.
The result is generalized in \cite{de2022quasi} to the quasi-static transform: if the reservoir densities are changing slowly at a time scale that is larger than the hyperbolic one, the profile evolves with the corresponding quasi-static solution \cite{marchesani2021quasi}.

The motivation of this article is to study the hydrodynamic limit for a hyperbolic system disturbed by a nonlocal external field.
We are particularly interested in the macroscopic behavior when the perturbation is extremely strong at the boundary.
Consider the ASEP on $\{1,\ldots,N-1\}$ where particles are created (resp.~annihilated) at each site $i$ with rate $V_i\rho_i$ (resp.~$V_i(1-\rho_i)$).
Assume two profiles $(V,\rho):(0,1)\to\R_+\times[0,1]$ such that
\begin{equation}\label{eq00}
	(V_i,\rho_i)=(V,\rho) \left( \tfrac iN \right), \quad \lim_{x\to0+} V(x) = \lim_{x\to1-} V(x) = \infty.
\end{equation}
In other words, a reservoir of density $\rho_i$ is placed at each site $i$, and the system exchanges particles with it with frequency $V_i$ that is growing rapidly near the boundaries.
When the exclusion dynamics is accelerated by $N$, the density profile shall evolve with the $L^\infty$ entropy solution to the following quasilinear balance law in the $[0,1]$-interval:
\begin{equation}\label{eq01}
	\partial_tu + \partial_x(u(1-u)) + G(x,u) = 0, \quad G=V(x)(u-\rho(x)),
\end{equation}
with proper boundary conditions.
We prove in Theorem \ref{thm:hydrodynamics} that, when the integrals of $V$ are infinity around both $0$ and $1$, the boundary conditions are $u|_{x=0}=\rho(0)$, $u|_{x=1}=\rho(1)$.
In sharp contrast to the equations obtained in \cite{Baha12, xu2022hydrodynamics, xu2021hydrodynamic}, the boundary values of $u$ are fixed by $\rho$ in a weak sense, see Proposition \ref{prop:boundary_continuity}.
Hence, any shock wave is attenuated while approaching the boundaries, and no boundary layer is observable at any positive macroscopic time.
A consequence of the hydrodynamic limit is the $L^1$-weak continuity in time of the entropy solution obtained in Corollary \ref{cor:regularity}.

The term $G$ in \eqref{eq01} acts as a source (resp.~sink) where $u$ is less (resp.~greater) than $\rho$, so it can be viewed as a relaxation scheme to the profile $\rho$.
When $\rho$ is a constant, it is a conservation system with relaxation introduced in \cite{Liu87}, with the first component degenerated to a stationary solution.
Such system is widely used to model non-equilibrium transport in kinetic theory and fluid dynamics.
In our situation, the entropy solution to the initial-boundary problem of \eqref{eq01} is constructed in different ways depending on the integrability of $V$, see Definition \ref{def:ent-sol-2} and \ref{def:ent-sol-1}.
We focus on the non-integrable case and discuss the integrable case briefly in Section \ref{sec:discussion}.

%The proof is based on the compensated compactness method introduced by Fritz \cite{Fritz04}. Compared with \cite{xu2022hydrodynamics,xu2021hydrodynamic}, we need to prove an additional microscopic energy bound, which requires a delicate analysis on the entropy of the microscopic dynamics. 
%
The proof in this article is proceeded in two main steps.
First, we prove that in the space-time scaling limit, the empirical Young measure of the particle field is concentrated on the space of Dirac-type Young measures.
Then, we show that the limit is a measure-valued entropy solution to \eqref{eq01} with proper boundary conditions.
The hydrodynamic limit then follows from the uniqueness of the entropy solution.
Both steps are proved through delicate analyses of the \emph{microscopic entropy production} associated with Lax entropy--flux pairs.
%In the second step, we also need to verify an energy estimate via a delicate analysis using the relative entropy.% method \cite{yau1991relative}.

The use of Young measure and microscopic entropy production is present in the seminal paper \cite{rezakhanlou91}.
It is combined with the \emph{compensated compactness method} to prove the concentration property of the Young measure in \cite{Fritz04, fritz2004derivation}.
To use this method, additional oscillating dynamics is added to ASEP to create microscopic viscosity.
Finally, we point out that although the process studied in this article is attractive, we cannot apply the coupling argument used in \cite{Baha12} because the invariant measure is not product in general.

\section{Model and Results}

\subsection{Model}\label{sec:model}
For a scaling parameter $N \in \N_+$, consider the configuration space
\[\Omega_N := \big\{\eta=(\eta_i)_{0 \leq i \leq N}, \eta_i \in \{0,1\}\big\}.\]
The dynamics on $\Omega_N$ consists of three parts: the nearest-neighbor asymmetric exclusion, the external Glauber field and the boundary dynamics.
The exclusion is generated by
\begin{equation*}
	L_{\rm exc} f (\eta) =  \sum_{i=0}^{N-1} \left( c_{i,i+1}(\eta)+\frac{\sigma_N}2 \right) \Big[f (\eta^{i,i+1}) - f(\eta)\Big],
\end{equation*}
for any function $f$ on $\Omega_N$, where, for constant $p\in(\tfrac12,1]$,
\begin{equation}
	c_{i,i+1} (\eta) = p\eta_i (1-\eta_{i+1}) + (1-p) \eta_{i+1} (1-\eta_i),
\end{equation}
$\sigma_N$ is a parameter that grows to infinity slower than $N$, and $\eta^{i,i+1}$ is the configuration obtained from $\eta$ by swapping the values of $\eta_i$ and $\eta_{i+1}$.
The factor $\sigma_N$ stands for a strong microscopic viscosity, which is necessary for the technique used in Section \ref{sec:compensated compactness}.
The Glauber dynamics is generated by
\begin{equation*}
	L_{\rm G} f (\eta)= \frac1N\sum_{i=1}^{N-1} c_{i,{\rm G}} (\eta)  \Big[ f (\eta^{i}) - f(\eta)\Big],
\end{equation*}
where, for parameters $V_i>0$ and $\rho_i\in(0,1)$,
\begin{equation}\label{eq:c_iG}
	c_{i,{\rm G}}(\eta):=V_i[\rho_i(1-\eta_i)+(1-\rho_i)\eta_i],
\end{equation}
and $\eta^{i}$ is the configuration obtained from $\eta$ by flipping the value of $\eta_i$.
Finally, the sites $i=0$ and $N$ are attached to two extra birth-and-death dynamics, interpreted as boundary reservoirs.
The corresponding generator reads
\begin{equation*}
	L_{\rm bd}f(\eta) = c_0(\eta) \Big[ f (\eta^{0}) - f(\eta)\Big] + c_{N} (\eta)  \Big[ f (\eta^{N}) - f(\eta)\Big],
\end{equation*}
where, for boundary rates $c_{\rm in}^\pm$, $c_{\rm out}^\pm \ge 0$,
\begin{equation}\label{eq:c_bd}
	c_0 (\eta) = c_{\rm in}^- (1-\eta_0) + c_{\rm out}^- \eta_0, \quad  c_N (\eta) =  c_{\rm in}^+(1-\eta_N) + c_{\rm out}^+ \eta_N.
\end{equation}

Assume two profiles $V:(0,1)\to\R_+$ and $\rho:[0,1]\to(0,1)$ such that $V_i=V(\tfrac iN)$ and $\rho_i=\rho(\tfrac iN)$ for $i=1$, ..., $N-1$.
Suppose that $V\in\mathcal C^1((0,1);\R_+)$, $V\to+\infty$ as $x\to0$, $1$, and $\rho\in\mathcal C^1([0,1];(0,1))$ with Lipschitz continuous $\rho'$.
In particular,
\begin{equation}\label{uniform_bound}
	\inf_{x\in(0,1)} V(x)>0, \quad \inf_{x\in[0,1]} \rho(x) > 0, \quad \sup_{x\in[0,1]} \rho(x) < 1.
\end{equation}
The generator of the process then reads
\begin{equation}\label{eq:generator}
	L_N =  N\big(L_{\rm exc} + L_{\rm G} + L_{\rm bd}\big),
\end{equation}
where the factor $N$ corresponds to the hyperbolic time scale.

\subsection{Scalar balance law in a bounded domain}
In this part, we introduce the partial differential equation that is obtained in the hydrodynamic limit for the model defined in the previous section.
Let
\begin{equation}
	J(u) := (2p-1) u (1-u), \quad G(x,u) := V(x) (u-\rho(x))
\end{equation}
be the macroscopic flux and the source term corresponding to $L_{\rm exc}$ and $L_{\rm G}$, respectively.
Given measurable function $u_0:(0,1)\to[0,1]$, consider the following balance law: for $(t,x)\in\Sigma:=\R_+\times(0,1)$,
\begin{equation}\label{hyperbolic_pde}
	\partial_{t} u(t,x)+\partial_{x}[J(u(t,x))]+G(x,u(t,x)) =0, \quad u|_{t=0}=u_0,
\end{equation}
with proper boundary conditions that will be specified later.

The weak solution to \eqref{hyperbolic_pde} is in general not unique, so we are forced to consider the \emph{entropy solution}.
Recalling \eqref{eq00}, our aim is to examine the case when the strength of the source is extremely strong at the boundaries.
We see in Definition \ref{def:ent-sol-2} and \ref{def:ent-sol-1} that the definition of entropy solution is different when $V$ is integrable or not at the boundaries.

We begin with the case that $V$ is \emph{non-integrable} at $0$ and $1$, i.e., for any small $y$,
\begin{equation}\label{v_condition}
	\int_0^y V(x)\,dx = +\infty, \qquad \int_{1-y}^1 V(x)\,dx = +\infty.
\end{equation}
Recall that a \emph{Lax entropy--flux pair} of \eqref{hyperbolic_pde} is a pair of functions $f$, $q \in\mathcal C^2(\R)$ such that $f''\ge0$ and $q'=J'f'=(2p-1)(1-2u)f'(u)$ for all $u\in\R$.

\begin{definition}\label{def:ent-sol-2}
Suppose that $V$ satisfies \eqref{v_condition}.
We call $u=u(t,x)$ an entropy solution to \eqref{hyperbolic_pde} with the compatible boundary conditions
\begin{equation}\label{eq:bl-bd}
	u(\cdot,0)=\rho(0), \qquad u(\cdot,1)=\rho(1),
\end{equation}
if $u:\Sigma\to[0,1]$ is  measurable and satisfies the generalized entropy inequality
	\begin{equation}\label{entropy_inequality}
		\begin{aligned}
			\int_{0}^{1} f(u_{0}) \varphi(0,\cdot)\,dx+\iint_{\Sigma}\big[f (u) \partial_{t} \varphi+q (u) \partial_{x} \varphi\big]\,d x d t\\
			\geq \iint_{\Sigma} f'(u)V(x)(u-\rho) \varphi\,d x d t.
		\end{aligned}
	\end{equation}
	for any Lax entropy--flux pair $(f,q)$ and any $\varphi \in \mathcal{C}_c^2 (\R \times (0,1))$, $\varphi \geq 0$.
%if the following conditions hold.
%\begin{enumerate}[(i)]
%	\item $u:\Sigma\to[0,1]$ is a measurable function that satisfies the energy bound
%	\begin{equation}\label{energy_bound}
%		\int_{0}^{T} \int_{0}^{1} V(x)[u(t, x)-\rho(x)]^{2}\,d x d t<\infty, \quad \forall\,T>0.
%	\end{equation}
%	\item For any Lax entropy--flux pair $(f,q)$ and any $\varphi \in \mathcal{C}_c^2 (\R \times (0,1))$ such that $\varphi \geq 0$, $u$ satisfies the generalized entropy inequality:
%	\begin{equation}\label{entropy_inequality}
%		\int_{0}^{1} f(u_{0}) \varphi(0,\cdot)\,dx+\iint_{\Sigma}\big[f (u) \partial_{t} \varphi+q (u) \partial_{x} \varphi\big]\,d x d t \geq \iint_{\Sigma} f'(u) G(\cdot, u) \varphi\,d x d t.
%	\end{equation}
%\end{enumerate}  
\end{definition}

\begin{remark}
\label{rem:energy_bound}
	When $\rho \in C^1$, $u$ in Definition \ref{def:ent-sol-2} satisfies the energy estimate
	\begin{equation}\label{energy_bound}
		\int_{0}^{T} \int_{0}^{1} V(x)[u(t, x)-\rho(x)]^{2}\,d x d t<\infty, \quad \forall\,T>0.
	\end{equation}
	Indeed, suppose that $\rho$ is smooth.
	For any $\varepsilon>0$, choose $\psi_\varepsilon \in \mathcal{C}_c^\infty((0,1))$ such that $\psi_\varepsilon(x)\in[0,1]$, $\psi_\varepsilon|_{[\varepsilon,1-\varepsilon]}\equiv1$ and $|\psi'_\varepsilon(x)|\le2\varepsilon^{-1}$.
	Fixing any $\phi \in \mathcal{C}_c^\infty(\R)$ such that $\phi\ge0$ and applying \eqref{entropy_inequality} with $f_1=\tfrac12u^2$, $\varphi_1=\phi(t)\psi_\varepsilon(x)$ and $f_2=-u$, $\varphi_2=\varphi_1\rho$ respectively, we obtain the upper bound
	\begin{align*}
		&\iint_\Sigma V(x)[u(t,x)-\rho(x)]^2\phi(t)\psi_\varepsilon(x)dxdt\\
		=\;&\iint_\Sigma f'_1(u)V(x)(u-\rho)\varphi_1\,dxdt + \iint_\Sigma f'_2(u)V(x)(u-\rho)\varphi_2\,dxdt\\
		\le\;&\phi(0)\int_0^1 \big[f_1(u_0)-u_0\rho\big]\psi_\varepsilon\,dx\\
		&+ \iint_\Sigma \big[f_1(u)\partial_t\varphi_1+q_1(u)\partial_x\varphi_1 - u\partial_t\varphi_2-J(u)\partial_x\varphi_2\big]dxdt,
\end{align*}
where $q_1$ is the flux corresponding to $f_1$.
	Since $|\psi_\varepsilon|\le1$, the first term on the right-hand side is bounded by $|\phi|_\infty\|f_1(u_0)-u_0\rho\|_{L^\infty}$.
	The second term reads
	\begin{align*}
		&\iint_\Sigma \big[(f_1(u)-u\rho)\partial_t\varphi_1 + (q_1(u)-J(u)\rho)\partial_x\varphi_1 - J(u)\rho'\varphi_1\big]dxdt\\
		\le\;&C \iint_\Sigma \big(|\phi'(t)\psi_\varepsilon(x)| + |\psi'_\varepsilon(x)\phi(t)| + |\rho'(x)\phi(t)\psi_\varepsilon(x)|\big)dxdt,
		\end{align*}
	where $C=\|f_1(u)-u\rho\|_{L^\infty} + \|q_1(u)-J(u)\rho\|_{L^\infty} + \|J(u)\|_{L^\infty}$.
	Since $|\psi_\varepsilon|\le1$, $|\psi'_\varepsilon|\le2\varepsilon^{-1}$ and is non-zero if and only if $x\in(0,\epsilon)\cup(1-\epsilon,1)$, it is bounded by $C_\phi(1+|\rho'|_\infty)$ with a constant $C_\phi$ that is independent of $\varepsilon$.
	Taking $\varepsilon\to0$ and using monotone convergence theorem,
	\[
		\iint_\Sigma V(x)[u(t,x)-\rho(x)]^2\phi(t)dxdt \le C_\phi(1+|\rho'|_\infty).
	\]
Since $\phi \in \mathcal{C}_c^\infty(\R;\R_+)$ is arbitrary, \eqref{energy_bound} holds for any finite $T>0$.
By standard argument of compactness, the estimate can be extended to any $\rho \in C^1([0,1])$.
\end{remark}

\begin{remark}
\label{rem:boundary_value}
If $u$ is continuous in space, \eqref{energy_bound} together with \eqref{v_condition} implies that $u(t,0)=\rho(0)$ and $u(t,1)=\rho(1)$ for almost all $t>0$.
Hence, \eqref{eq:bl-bd} turns out to be the reasonable choice of the boundary conditions, see also Proposition \ref{prop:boundary_continuity} below.
\end{remark}

The following uniqueness criteria is taken from \cite[Theorem 2.12]{Xu24}.

\begin{proposition}\label{prop:uniq}
Assume further that
\begin{align}
	\label{v_condition_1}
	&\limsup_{y \rightarrow 0+} \frac{1}{y^2} \int_{0}^y \left[ \frac{1}{V(x)} + \frac{1}{V(1-x)} \right] \,dx < + \infty,\\
	\label{v_condition_2}
	&\lim_{y \rightarrow 0+} \left[ \int_0^y V(x) \big[\rho(x) - \rho(0)\big]^2 dx + \int_{1-y}^1 V(x) \big[\rho(x) - \rho(1)\big]^2 dx \right] = 0.
	\end{align}
Then, there is at most one function $u \in L^\infty(\Sigma)$ that fulfills Definition \ref{def:ent-sol-2}.
\end{proposition}

\begin{remark}
Suppose that $V>0$ satisfies \eqref{v_condition_1}.
By \CS inequality,
\[\liminf_{y\to0+} \int_0^y V(x)dx \ge \liminf_{y\to0+} \bigg(\frac1{y^2}\int_0^y \frac1{V(x)}dx\bigg)^{-1} > 0,\]
which means that $V$ is not integrable at $0$.
The same argument holds for the integration on $(1-y,1)$.
Therefore, \eqref{v_condition_1} contains the non-integrable condition \eqref{v_condition}.
\end{remark}

Now we turn to the \emph{integrable case}: $V \in L^1((0,1))$.
The next definition is first introduced by F.~Otto, see \cite[Eq.~9]{Otto96}.

\begin{definition}
We call $(F, Q) \in \mathcal{C}^{2}([0,1]^2 ; \mathbb{R}^{2})$ a boundary entropy--flux pair if
\begin{enumerate}[(i)]
	\item for all $k \in [0,1]$, $(F,Q) (\cdot, k)$ is a Lax entropy-flux pair,  \emph{i.e.}, $\partial_uQ(\cdot,k)=J'\partial_uF(\cdot,k)$;
	\item for all $k \in [0,1]$, $F(k, k)=\partial_{u} F(u, k)|_{u=k}=Q(k, k)=0$.
\end{enumerate}
\end{definition}

For $V$ integrable, the definition of entropy solution is the same as \cite[Proposition 2]{Otto96} for $V\equiv0$ and \cite[Definition 1]{Martin07} for $V$ bounded and smooth.

\begin{definition}\label{def:ent-sol-1}
Let $\alpha$, $\beta\in[0,1]$ be two constants and suppose that $V \in L^1 ((0,1))$.
We call $u=u(t,x)$ an entropy solution to \eqref{hyperbolic_pde} with the boundary conditions given by
\begin{equation}
	u(t,0)=\alpha, \qquad u(t,1)=\beta,
\end{equation}
if $u: \Sigma\to[0,1]$ is a measurable function such that for any boundary entropy--flux pair $(F,Q)$, any $k \in [0,1]$, and any $\varphi \in \mathcal{C}_c^2 (\R^2)$ such that $\varphi \geq 0$,
\begin{equation}
	\begin{aligned}
		&\int_{0}^{1} f_k(u_{0}) \varphi(0,\cdot)\,dx + \iint_{\Sigma} \big[f_k(u) \partial_{t} \varphi + q_k(u) \partial_{x} \varphi\big]\,d x d t
\\
		&\quad\geq \iint_{\Sigma} f'_k(u) V(x)(u-\rho) \varphi\,d x d t -  \int_{0}^{\infty}\big[f_k(\beta)\varphi(\cdot, 1)+ f_k(\alpha)\varphi(\cdot, 0)\big]\,d t,
	\end{aligned}
\end{equation}
where $(f_k,q_k):=(F,Q)(\cdot,k)$.
%the following conditions.
%\begin{enumerate}[(i)]
%	\item For any boundary entropy--flux pair $(F,Q)$ and any $\phi\in\mathcal C_c(\R)$ such that $\phi>0$,
%	\begin{equation}\label{eq:bd_layers}
%		\esslim_{x\to0} \int_{\R_+} Q(u(\cdot,x),\alpha)\phi\,dt \le 0, \quad \esslim_{x\to1} \int_{\R_+} Q(u(\cdot,x),\beta)\phi\,dt \ge 0.
%	\end{equation}	
%	\item For any Lax entropy--flux pair $(f,q)$ and any $\varphi\in\mathcal C_c(\R\times(0,1))$ such that $\varphi \geq 0$, $u$ satisfies the generalized entropy inequality \eqref{entropy_inequality}.
%\end{enumerate}  
\end{definition}

%In sharp contrast to the non-integrable case, \eqref{eq:bd_layers} provides a set of possible values of the trace of $u$ at the boundaries.
%In other words, when $V \in L^1((0,1))$ the boundary conditions are in formal, and the entropy solution is allowed to present discontinuities, namely the boundary layers, due to the lack of viscosity in \eqref{hyperbolic_pde}.

%\begin{remark}
%Mixed boundary.
%\end{remark}

Since the integrable case is not the focus of this paper, we omit the uniqueness and other properties and refer to \cite{Xu24} and the references therein.

\subsection{Hydrodynamic limit}
Let $\{\eta^N(t)\in\Omega_N;t\ge0\}$ be the Markov process generated by $L_N$ in \eqref{eq:generator} and initial distribution $\mu_N$.
Through this article, the superscript $N$ in $\eta^N$ is omitted when there is no confusion.
Denote by $\P_{\mu_N}$ the distribution of $\eta(\cdot)$ on $\mathcal D([0,\infty),\Omega_N)$, the space of all c\`adl\`ag paths on $\Omega_N$, and by $\E_{\mu_N}$ the expectation of $\P_{\mu_N}$.

Suppose that the sequence of $\mu_N$ is associated with a measurable function $u_0:(0,1)\to[0,1]$ in the following sense: for any $\psi\in\mathcal C(\R)$,
\begin{equation}\label{eq:initial}
	\lim_{N \rightarrow \infty} \mu_N \left\{ \left| \frac{1}{N} \sum_{i=0}^N \eta_i(0)\psi \left( \tfrac iN \right) - \int_0^1 u_0 (x) \psi (x) dx \right| > \delta \right\} = 0, \quad \forall\,\delta>0.
\end{equation}
%a stronger condition on $\hat\eta_{i,K}(0)$ for some $K$ such that $\max\{(N\sigma_N)^{\frac13}, \sigma_N^2N^{-1}\} \ll K \ll \sigma_N$.
Our main result shows that in the non-integrable case,  the empirical density of the particles converges, as $N\to\infty$, to the entropy solution to \eqref{hyperbolic_pde} and \eqref{eq:bl-bd}.

\begin{theorem}\label{thm:hydrodynamics}
	Assume \eqref{v_condition_1}, \eqref{v_condition_2} and \eqref{eq:initial}.
	Also assume that
	\begin{equation}\label{sigma_condition}
		\lim_{N\to\infty} N^{-1}\sigma_N^2 = \infty, \qquad \lim_{N\to\infty} N^{-1}\sigma_N = 0.
	\end{equation}
	Then, for any $\psi \in \mathcal{C}(\R)$ and  almost all $t > 0$,
	\begin{equation}
	\lim_{N\to\infty} \P_{\mu_N} \left\{ \left| \frac{1}{N} \sum_{i=0}^N \eta_i (t) \psi \left( \tfrac{i}{N} \right) - \int_0^1 u(t,x)\psi (x) dx \right| >\delta \right\} = 0, \quad \forall\,\delta>0,
	\end{equation}
	where $u$ is the unique entropy solution in Definition \ref{def:ent-sol-2}.
%to \eqref{hyperbolic_pde} with the boundary conditions given by 
%\begin{itemize}
%	\item $\alpha =u_-, \;\beta =u_+$ if \;$V \in L^1((0,1))$ and  $a>0$;
%	\item $\alpha =0, \;\beta =1$ if \;$V \in L^1((0,1))$ and  $a<0$;
%	\item 
%$\alpha = \rho (0),\; \beta = \rho (1)$% if \; $V \notin L^1((0,1))$ and $ a\in\R$.
%\end{itemize}}%
\end{theorem}

\begin{remark}
Below we list two important remarks concerning Theorem \ref{thm:hydrodynamics}.
\begin{enumerate}[(i)]
\item We assume \eqref{v_condition_1} and \eqref{v_condition_2} only for the uniqueness in Proposition \ref{prop:uniq}.
If $V$ satisfies only \eqref{v_condition}, our argument proves that the empirical distribution of $\eta^N$ is tight and all limit points are concentrated on the possible entropy solutions.
\item Observe that the rates in $L_{\rm bd}$ do not appear in the limit.
Indeed, let $F_\epsilon(\eta) = \sum_{0 \le i \le \epsilon N} \eta_i$ be the cumulative mass on $\{\eta_0,\ldots,\eta_{[\epsilon N]}\}$.
Then,
\begin{align*}
L_N F_\epsilon(\eta) =\,&N\big[c_{\rm in}^- - (c_{\rm in}^-+c_{\rm out}^-)\eta_0\big] -NJ_{[\epsilon N]}\\
&- N(1-p+\sigma_N)(\eta_{[\epsilon N]}-\eta_{[\epsilon N]+1}) + \sum_{1 \le i \le \epsilon N} V_i(\rho_i-\eta_i),
\end{align*}
where $J_i=(2p-1)\eta_i(1-\eta_{i+1})$.
From \eqref{v_condition}, $\sum_{1 \le i \le \epsilon N} V_i \gg N$.
Hence, to make the contribution of the last term be of order $\mathcal O(N)$, the mass density of the boundary block $\{1,\ldots,\epsilon N\}$ should be near to $\rho(0)$.
\end{enumerate}
\end{remark}

As a corollary of Theorem \ref{thm:hydrodynamics}, the regularity of the entropy solution is improved.

\begin{corollary}\label{cor:regularity}
Assume \eqref{v_condition_1} and \eqref{v_condition_2}.
Let $u$ be the unique entropy solution to \eqref{hyperbolic_pde} and \eqref{eq:bl-bd} in Definition \ref{def:ent-sol-2}.
Then,
\[u \in L^\infty((0,\infty)\times(0,1)) \cap \mathcal C([0,\infty);L^1),\]
where $L^1=L^1((0,1))$ is endowed with the weak topology.
In particular, the convergence in Theorem \ref{thm:hydrodynamics} holds for all $t>0$.
\end{corollary}

Under \eqref{v_condition_1}, the macroscopic density near the boundary is prescribed by the reservoir in the following sense: for any $t>0$,
\begin{equation}\label{eq:boundary_continuity}
  \lim_{y\to0+} \lim_{N\to\infty} \int_0^t \frac{1}{yN} \sum_{i=0}^{\lfloor yN \rfloor} \eta_i (s)ds = t\rho(0) \quad \text{in $\P_{\mu_N}$-probability},
\end{equation}
and similarly for the right boundary.
Indeed, for any $t>0$, $y\in(0,1)$ and $\delta>0$,
\begin{equation*}
  \begin{aligned}
    &\frac1y\int_0^t \int_0^y |u(s,x)-\rho(x)|\,dxds\\
    \le\,&\frac1{4\delta}\int_0^t \int_0^y V(x)(u-\rho)^2dxds + \frac{t\delta}{y^2}\int_0^y \frac1{V(x)}dx.
  \end{aligned}
\end{equation*}
Taking $y\to0+$, \eqref{energy_bound} together with \eqref{v_condition_1} yields that
\begin{equation*}
  \lim_{y\to0+} \frac1y\int_0^t \int_0^y |u(s,x)-\rho(x)|\,dxds \le Ct\delta.
\end{equation*}
As $\delta$ is arbitrary, the limit is $0$.
Recall that $\rho$ is continuous, so we have
\begin{equation*}
  \lim_{y\to0+} \frac1y\int_0^t \int_0^y |u(s,x)-\rho(0)|\,dxds = 0, \quad \forall\,t>0.
\end{equation*}
Combining this with Theorem \ref{thm:hydrodynamics}, we obtain \eqref{eq:boundary_continuity} for all positive time $t$.
These limits can be derived directly from the microscopic dynamics by imposing a slightly stronger growth condition on $V$, see the next proposition.

\begin{proposition}\label{prop:boundary_continuity}
Suppose that $V$ satisfies the following condition:
\begin{equation}\label{v_condition_3}
\lim_{y\to0+} \Big\{y\inf_{x\in(0,y)} V(x)\Big\} = \infty.
\end{equation}
Then, \eqref{eq:boundary_continuity} and the similar limit for the right boundary hold for all $t > 0$.
\end{proposition}

\begin{example*}
Fix some $\gamma>0$, $\rho_0$ and $\rho_1\in(0,1)$.
By taking
\[V(x)=\frac1{x^\gamma}+\frac1{(1-x)^\gamma}, \quad \rho(x)=\frac{\rho_0(1-x)^\gamma+\rho_1x^\gamma}{(1-x)^\gamma+x^\gamma},\]
we obtain the source term given by
\[G(x,u)=\frac{u-\rho_0}{x^\gamma} + \frac{u-\rho_1}{(1-x)^\gamma}.\]
In this case, the dynamics of $L_{\rm G}$ can be interpreted as two infinitely extended reservoirs \cite{BCGS23, bernardin20192low, BGJ21} placed respectively at the sites $\{-1,-2,\ldots,\}$ and $\{N+1,N+2,\ldots\}$.
When $\gamma\ge1$, $V$ satisfies \eqref{v_condition_1}, so the hydrodynamic limit can apply.
\end{example*}

\subsection{Discussion on the integrable case}
\label{sec:discussion}
When $V \in L^1 ((0,1))$, we expect that Theorem \ref{thm:hydrodynamics} holds with the entropy solution in Definition \ref{def:ent-sol-1}.
Since the dynamics of $L_{\rm G}$ is no more dominating at the boundaries, the boundary data $(\alpha,\beta)$ may depend on $c_{\rm in}^\pm$, $c_{\rm out}^\pm$ as well as $V$, $\rho$.
In particular when $V=0$, we expect that $\alpha$, $\beta$ are determined by
\[J(\alpha) = c_{\rm in}^-(1-\alpha) - c_{\rm out}^-\alpha, \qquad J(\beta) = c_{\rm out}^-\beta - c_{\rm in}^+(1-\beta).\]
This is proved in \cite{Baha12} for microscopic dynamics \emph{without extra symmetric regularization} and the special choice of reservoirs such that
\[c_{\rm in}^-=p\alpha, \quad c_{\rm out}^-=(1-p)(1-\alpha), \quad c_{\rm in}^+=(1-p)\beta, \quad c_{\rm out}^+=p(1-\beta).\]
We underline that the problem remains open for general reservoirs even when $V=0$.

%and one speeds up the boundary dynamics $L_{\rm bd}$ by $N^{1+a}$ for some $a \in \R$ instead of $N$,  then the law of large numbers for the empirical measure of the process $\eta^N$ still holds, and the hydrodynamic equation $u =u(t,x)$ is the unique entropy solution (see Definition \ref{def:ent-sol-1}) to \eqref{hyperbolic_pde} with the boundary conditions given by 
%	\begin{enumerate}
%			\item $\alpha = \tfrac{c_{\rm in}^-}{c_{\rm in}^- + c_{\rm out}^-}, \;\beta = \tfrac{c_{\rm in}^+}{c_{\rm in}^+ + c_{\rm out}^+}$ if   $a>0$;
%			\item $\alpha =0, \;\beta =1$ if   $a<0$.	
%	\end{enumerate}
%We underline that the case $V \in L^1 ((0,1)), a = 0$ remains open. Since the proof is the same as in \cite{xu2021hydrodynamic}, we omit it and focus on the non-integrable case in this article.
The situation is easier when further speed-up is imposed on the boundary reservoirs.
Let $V \in L^1((0,1))$ satisfy \eqref{eq00} and assume the compatibility conditions
\[\rho(0) = \frac{c_{\rm in}^-}{c_{\rm in}^- + c_{\rm out}^-}, \qquad \rho(1) = \frac{c_{\rm in}^+}{c_{\rm in}^+ + c_{\rm out}^+}.\]
Fix $a>0$ and consider the process generated by $L'_N := N(L_{\rm exc} + L_{\rm G} + N^aL_{\rm bd})$.
In this case, the hydrodynamic equation is still given by \eqref{hyperbolic_pde} and \eqref{eq:bl-bd}, but the solution should be understood in the sense of Definition \ref{def:ent-sol-1}.
This can be proved with the argument in \cite{xu2021hydrodynamic}.
%
%\noindent\textbf{Macroscopic stationary profile}.
%For open ASEP, the hydrostatic limit is studied in \cite{Baha12, DEHP93}: the stationary macroscopic profile is the stationary solution to the hydrodynamic equation.
%The stationary solution is determined by the boundary data through a variational rule that, the stationary flux is maximized in case of a density gradient with opposite sign to the drift of asymmetry, or minimized in the other case.
%It is generated in \cite{de2022quasi} to the quasi-static transform: when the boundary condition is changing at the time scale that is larger than the hyperbolic one, the macroscopic density profile evolves with the corresponding quasi-static solution \cite{marchesani2021quasi}.
%In our situation, the relaxation scheme prevents the system to formulate any stable shock.
%Hence, we expect that the macroscopic stationary profile is given by a continuous function connecting $\rho(0)$ with $\rho(1)$.

\section{Outline of the proof}

Hereafter, we fix an arbitrary $T>0$ and restrict the argument within the finite time horizon $[0,T]$.
Let $\mathcal M_+([0,1])$ be the space of finite, positive Radon measures on $[0,1]$, endowed with the weak topology.
Define the empirical distribution $\pi^N=\pi^N(t,dx)$ as
\begin{equation}\label{eq:empirical}
  \pi^N(t,dx) := \frac1N\sum_{i=0}^N \eta_i(t)\delta_{\frac iN}(dx), \quad \forall\,t\in[0,T],
\end{equation}
where $\delta_u(dx)$ stands for the Dirac measure at $u$.
Denote by $\mathcal D=\mathcal D([0,T];\mathcal M_+([0,1]))$ the space of c\`adl\`ag paths on $\mathcal M_+([0,1])$ endowed with the Skorokhod topology.
To prove Theorem \ref{thm:hydrodynamics}, it suffices to show that the distribution of $\pi^N$ on $\mathcal D$ converges weakly as $N\to\infty$ and the limit is concentrated on the single path $\pi(t,dx)=u(t,x)dx$.
However, to formulate the evolution equation \eqref{hyperbolic_pde} of $u$ we need a type of convergence that also applies to nonlinear functions.
The idea is to introduce the \emph{Young measure} corresponding to the mesoscopic block average, cf.~\cite{Fritz04, fritz2004derivation, rezakhanlou91} and \cite[Chapter 8]{klscaling}.

Let $\Sigma_T=(0,T)\times(0,1)$.
Recall that a Young measure on $\Sigma_T$ is a measurable map $\nu:\Sigma_T\to\mathcal P(\R)$, where $\mathcal P(\R)$ is the space of probability measures on $\R$ endowed with the topology defined by the weak convergence.
Denote by $\mathcal{Y}=\mathcal Y(\Sigma_T)$  the set of all Young measures on $\Sigma_T$, and by $\nu=\{\nu_{t,x};(t,x)\in\Sigma_T\}$ the element in $\mathcal Y$.
A sequence $\{\nu^n;n\ge1\}$ of Young measures is said to converge to $\nu\in\mathcal Y$ if for any bounded and continuous function $f$ on $\Sigma_T\times\R$,
\begin{equation}\label{eq:topology-young}
	\lim_{n\to\infty} \iint_{\Sigma_T} dxdt \int_\R f(t,x,\lambda)\nu_{t,x}^n(d\lambda) = \iint_{\Sigma_T} dxdt \int_\R f(t,x,\lambda)\nu_{t,x}(d\lambda).
\end{equation}
Any measurable function $u$ on $\Sigma_T$ is naturally viewed as a Young measure:
\begin{equation}\label{eq:delta-young}
  \nu_{t,x}(d\lambda):=\delta_{u(t,x)}(d\lambda), \quad \forall\,(t,x)\in\Sigma_T.
\end{equation}
Denote by $\mathcal Y_d$ the set of all $\nu\in\mathcal Y$ of this kind.

Hereafter, we fix some mesoscopic scale $K = K (N)$ such that
\begin{equation}\label{eq:mesoscopic}
   K \ll \sigma_N, \quad  N \sigma_N \ll K^3, \quad \sigma_N^2 \ll NK.
\end{equation}
The existence of such $K$ is guaranteed by \eqref{sigma_condition}.
For $\eta \in \Omega_N$ and $i=K$, ..., $N-K$, define the \emph{smoothly weighted block average} as
\begin{equation}\label{eq:block-average}
	\hat{\eta}_{i,K} := \sum_{j=-K+1}^{K-1} w_{j} \eta_{i-j}, \quad w_j := \frac{K-|j|}{K^2}.
\end{equation}
Consider the space-time empirical density
\begin{equation}\label{eq:empirical-young}
  u^N(t,x) := \sum_{i=K}^{N-K} \hat\eta_{i,K}(t)\chi_{N,i}(x), \quad \forall\,(t,x)\in\Sigma_T,
\end{equation}
where $\chi_{N,i} (\cdot)$ is the indicator function of the interval $[\tfrac{i}{N} - \tfrac{1}{2N}, \tfrac{i}{N} + \tfrac{1}{2N})$. 

\begin{lemma}[Tightness]\label{lem:tight}
Let $\Q_N$ be the distribution of $(\pi^N,\nu^N)$, where $\pi^N$ is defined in \eqref{eq:empirical} and $\nu^N$ is the Young measure corresponding to $u^N$ in \eqref{eq:empirical-young} in the sense of \eqref{eq:delta-young}.
Then, the sequence of $\mathbb Q_N$ is tight with respect to the product topology on $\mathcal D\times\mathcal Y$.
\end{lemma}

Let $\Q$ be a limit point of $\Q_N$.
With some abuse of notations, we denote the subsequence converging to $\Q$ still by $\Q_N$.
Below we characterize $\Q$ by three propositions.

\begin{proposition}\label{prop:q}
The following holds for $\Q$-almost every $(\pi,\nu)$.
\begin{enumerate}
\item[$(\romannumeral1)$] $\pi(t,dx)=\varpi(t,x)dx$ for every $t\in[0,T]$ with some $\varpi(t,\cdot) \in L^1((0,1))$, and $t\mapsto\varpi(t,\cdot)$ is a continuous map with respect to the weak topology of $L^1$.
\item[$(\romannumeral2)$] $\nu_{t,x}([0,1])=1$ for almost all $(t,x)\in\Sigma_T$.
\item[$(\romannumeral3)$] $\varpi(t,x)=\int \lambda\nu_{t,x}(d\lambda)$ for almost all $(t,x)\in\Sigma_T$.
\end{enumerate}
\end{proposition}

\begin{proposition}\label{Q_dirac_property}
$\Q(\mathcal D \times \mathcal Y_d)=1$, where $\mathcal Y_d$ is the set of delta-Young measures in \eqref{eq:delta-young}.
\end{proposition}

To state the last proposition, define the \emph{entropy production}
\begin{equation}\label{eq:ent-prod}
	X^{(f,q)} (\nu,\varphi) := - \iint_{\Sigma_T} dxdt \left[ \partial_t\varphi\int_\R f\,d\nu_{t,x} + \partial_x\varphi\int_\R q\,d\nu_{t,x} \right],
\end{equation}
for $\nu\in\mathcal Y$, $\varphi \in \mathcal{C}^1 (\R^2)$ and Lax entropy--flux pair $(f,q)$.

\begin{proposition}\label{prop:q-mve}
It holds $\Q$-almost surely that
%\begin{equation}\label{eq:micro-energy-bd}
%  \iint_{\Sigma_T} dxdt \left[ V(x)\int_\R \big[\lambda-\rho(x)\big]^2d\nu_{t,x} \right] < \infty.
%\end{equation}
%Moreover, for any Lax entropy--flux pair $(f,q)$ and test function $\varphi\in\mathcal C_c^2([0,T)\times(0,1))$ such that $\varphi\ge0$, we have $\Q$-almost surely that
\begin{equation}\label{eq:micro-entropy-inequality}
  X^{(f,q)}(\nu,\varphi) + \iint_{\Sigma_T} dxdt \left[ \varphi\int_\R f'(\lambda)G(x,\lambda)\nu_{t,x}(d\lambda) \right] \le \int_0^1 f(u_0)\varphi(0,x)dx,
\end{equation}
for any Lax entropy--flux pair $(f,q)$ and any $\varphi\in\mathcal C_c^2([0,T)\times(0,1))$ such that $\varphi\ge0$, where $G(x,\lambda)=V(x)[\lambda-\rho(x)]$.
\end{proposition}

\begin{remark}
Similarly to Remark \ref{rem:energy_bound}, one can obtain a measure-valued energy bound: it holds $\Q$-almost surely that
\begin{equation}\label{eq:micro-energy-bd}
  \iint_{\Sigma_T} dxdt \left[ V(x)\int_\R \big[\lambda-\rho(x)\big]^2d\nu_{t,x} \right] < \infty.
\end{equation}
This can be derived directly from the microscopic dynamics, see Section \ref{sec:mve}.
\end{remark}

\begin{remark}
The arguments we used to prove Lemma \ref{lem:tight} and Proposition \ref{prop:q}, \ref{Q_dirac_property} and \ref{prop:q-mve} do apply to all $V \in \mathcal C^1((0,1))$, bounded or unbounded. However, only in the non-integrable case are they sufficient to identify the limit equation.
\end{remark}

We organize the remaining contents as follows.
Lemma \ref{lem:tight} and Proposition \ref{prop:q} are proved in Section \ref{sec:tight}.
Proposition \ref{Q_dirac_property} is proved in Section \ref{sec:compensated compactness} and Proposition \ref{prop:q-mve} is proved in Section \ref{subsec:micro-ent-ineq}.
With these results, the proofs of Theorem \ref{thm:hydrodynamics} and Corollary \ref{cor:regularity} are straightforward and are stated right below.
The direct proofs of Proposition \ref{prop:boundary_continuity} and \eqref{eq:micro-energy-bd} using the relative entropy method \cite{yau1991relative} are stated in Section \ref{subsec:micro-energy-bd} and \ref{subsec:micro-bd-cont}.

\begin{proof}[Proof of Theorem \ref{thm:hydrodynamics} and Corollary \ref{cor:regularity}]
Recall that $\Q$ is a probability measure on $\mathcal D\times\mathcal Y$.
In view of Proposition \ref{prop:q} $(i)$ and \ref{Q_dirac_property}, $\pi(t)=\varpi(t,x)dx$ and $\nu_{t,x}=\delta_{u(t,x)}$, $\Q$-almost surely.
Proposition \ref{prop:q} $(iii)$ then yields that $\Q$ is concentrated on the trajectories such that $\varpi=u$.

To prove Theorem \ref{thm:hydrodynamics}, we need to show that $u$, and hence $\varpi$, is the entropy solution to \eqref{hyperbolic_pde} and \eqref{eq:bl-bd}.
By Proposition \ref{prop:q} $(ii)$, $u(t,x)\in[0,1]$ so that $u \in L^\infty(\Sigma_T)$.
Furthermore, by substituting $\nu_{t,x}=\delta_{u(t,x)}(d\lambda)$ in Proposition \ref{prop:q-mve}, we obtain that $u$ % satisfies both \eqref{energy_bound} and \eqref{entropy_inequality}
satisfies the generalized entropy inequality in Definition \ref{def:ent-sol-2}.
The proof is then concluded by the uniqueness of the entropy solution, see Proposition \ref{prop:uniq}.

Finally, Corollary \ref{cor:regularity} follows directly from the argument above and the sample path regularity of $\varpi$ obtained in Proposition \ref{prop:q} $(i)$.
\end{proof}

We close this section with some useful notations.
For a function $\varphi=\varphi(t,x)$, let
\begin{equation}\label{eq:notion1}
  \varphi_i(t) := \varphi \left( t,\tfrac{i}{N} - \tfrac{1}{2N} \right), \quad \bar{\varphi}_i(t) := N \int_{\tfrac{i}{N} - \tfrac{1}{2N}}^{\tfrac{i}{N} + \tfrac{1}{2N}} \varphi (t,x) dx.
\end{equation}
Recall the mesoscopic scale $K=K(N)$ in \eqref{eq:mesoscopic}.
For a sequence $\{a_i;i=0,\ldots,N\}$, $\hat a_{i,K}$ stands for the smoothly weighted average in \eqref{eq:block-average}.
Since $K$ is fixed through the paper, we write $\hat a_i$ when there is no confusion.
We shall frequently use the notions of discrete gradient and Laplacian operators, which are defined as usual:
\begin{equation}\label{eq:notion2}
  \nabla a_i =a_{i+1} - a_i, \quad \nabla^* a_i = a_{i-1} - a_i, \quad \Delta a_i = a_{i+1} - 2 a_i + a_{i-1}.
\end{equation}
%We shall frequently use the following summation-by-parts formula:
%\begin{equation}\label{summation by parts}
%	\sum_{i=K}^{N-K} a_i \nabla b_i = \sum_{i=K}^{N-K} b_i \nabla^* a_i + a_{N-K} b_{N-K+1} - {\color{blue}a_K b_K}.
%\end{equation}
Notice that $\Delta=-\nabla\nabla^*=-\nabla^*\nabla$.

\section{Tightness}\label{sec:tight}

Recall that $\Q_N$ is the distribution of $(\pi^N,\nu^N)$ on $\mathcal D\times\mathcal Y$.
Here $\mathcal D=\mathcal D([0,T],\mathcal M_+([0,1]))$ is the space of c\`adl\`ag paths endowed with the Skorokhod topology and $\mathcal Y=\mathcal Y(\Sigma_T)$ is the space of Young measures endowed with the topology defined by \eqref{eq:topology-young}.

\begin{proof}[Proof of Lemma \ref{lem:tight}]
It suffices to show that both $\{\pi^N\}$ and $\{\nu^N\}$ are tight.
The coordinate $\nu^N$ is easy.
Since $u^N\in[0,1]$, for all $N$ we have $\Q_N\{\nu^N\in\mathcal Y_*\}=1$, where
\[\mathcal Y_* := \big\{\nu\in\mathcal Y \,\big|\, \nu_{t,x}([0,1])=1 \text{ for all } (t,x)\in\Sigma_T\big\}.\]
%From \cite[Proposition 4.1]{BertheV19}, 
Because $\mathcal Y_*$ is compact in $\mathcal Y$, $\{\nu^N\}$ is tight.

For the coordinate $\pi^N$, by \cite[Chapter 4, Theorem 1.3 \& Proposition 1.7]{klscaling}, we only need to show that for any $\psi\in\mathcal C([0,1])$, some constant $C_\psi$ and any $\delta>0$,
\begin{align}
  &\sup_N \Q_N \left\{ \sup_{t\in[0,T]} \big|\langle \pi^N(t),\psi \rangle\big| < C_\psi \right\} = 1,\\
  &\lim_{\varepsilon\downarrow0} \lim_{N\to\infty} \Q_N \left\{ \sup_{|t-s|<\varepsilon} \big|\langle \pi^N(t),\psi \rangle - \langle \pi^N(s),\psi \rangle\big| > \delta \right\} = 0,\label{eq:tight-2}
\end{align}
where $\langle\,\cdot,\cdot\,\rangle$ is the scalar product between $\mathcal M_+$ and $\mathcal C([0,1])$.
The first one is obvious:
\begin{equation}\label{eq:tight-3}
  \big|\langle \pi^N(t),\psi \rangle\big| = \frac1N \left| \sum_{i=0}^N \eta_i(t)\psi \left( \tfrac iN \right) \right| \le \frac1N\sum_{i=0}^N \left| \psi \left( \tfrac iN \right) \right|, \quad \forall\,t\in[0,T].
\end{equation}
For the second one, choose $\psi_*\in\mathcal C_c^2((0,1))$ such that $\|\psi-\psi_*\|_{L^1}<4^{-1}\delta$.
Note that
\[\big|\langle \pi^N(t),\psi-\psi_* \rangle - \langle \pi^N(s),\psi-\psi_* \rangle\big| \le \frac1N\sum_{i=0}^N \left| \psi \left( \tfrac iN \right) - \psi_* \left( \tfrac iN \right) \right| < \frac\delta2,\]
uniformly in $s$, $t$ and all sample paths.
Hence, it suffices to show \eqref{eq:tight-2} with $\psi$ replaced by $\psi_*$.
Without loss of generality, let $s<t$.
Then,
\begin{equation}\label{eq:tight-4}
  \langle \pi^N(t),\psi_* \rangle - \langle \pi^N(s),\psi_* \rangle = \int_s^t L_N\big[\langle \pi^N(r),\psi_* \rangle\big]dr + M_{N,\psi_*}(t-s),
\end{equation}
where $M_{N,\psi_*}$ is the Dynkin's martingale.
As $\psi_*$ is compactly supported, $\eta_0$ or $\eta_N$ does not appear in $\langle \pi^N(r),\psi_* \rangle$, so that
\begin{align*}
	L_N\big[\langle \pi^N,\psi_* \rangle\big] =\,&(1-p+\sigma_N) \sum_{i=1}^{N-1} \eta_i \left[ \psi_* \left( \tfrac{i+1}N \right) + \psi_* \left( \tfrac{i-1}N \right) - 2\psi_* \left( \tfrac iN \right) \right]\\
    	&+\,(2p-1)\sum_{i=0}^{N-1} \eta_i(1-\eta_{i+1}) \left[ \psi_* \left( \tfrac{i+1}N \right) - \psi_* \left( \tfrac iN \right) \right]\\
   	&+\,\frac1N\sum_{i=0}^N V \left( \tfrac iN \right) \left[ \eta_i-\rho \left( \tfrac iN \right) \right] \psi_* \left( \tfrac iN \right).
\end{align*}
Using the fact that $\sigma_N \ll N$ and $\psi_*\in\mathcal C_c^2((0,1))$, $L_N[\langle \pi^N,\psi_* \rangle]$ is uniformly bounded.
Therefore, the first term in \eqref{eq:tight-4} vanishes uniformly when $|t-s|\to0$.
We are left with the martingale in \eqref{eq:tight-4}.
Dy Dynkin's formula, the quadratic variation reads
\[\langle M_{N,\psi_*} \rangle(t-s) = \int_s^t \Big(L_N\big[\langle \pi^N,\psi_* \rangle^2\big] -  2 \langle \pi^N,\psi_* \rangle L_N\big[\langle \pi^N,\psi_* \rangle\big]\Big)\,dr.\]
Recall that $\psi_*\in\mathcal C_c^2((0,1))$, direct calculation shows that
\begin{align*}
L_N\big[\langle \pi^N,\psi_* \rangle^2\big] -  2 \langle \pi^N,\psi_* \rangle L_N\big[\langle \pi^N,\psi_* \rangle\big] = \frac{1}{N^2}\sum_{i=0}^{N-1} c_{i,{\rm G}} (\eta) (1-2\eta_i)^2 \psi_*^2 \left( \tfrac iN \right)\\
+\,\frac1N\sum_{i=0}^{N-1} \left( c_{i,i+1} (\eta) + \frac{\sigma_N}{2} \right) (\eta_i - \eta_{i+1})^2 \left[ \psi_* \left( \tfrac{i+1}N \right) - \psi_* \left( \tfrac iN \right) \right]^2.
\end{align*}
By \eqref{sigma_condition}, it is bounded from above by $CN^{-1}$.
Therefore,
\begin{equation}
  E^{\Q_N} \big[|\langle M_{N,\psi_*} \rangle(t-s)|\big] \le C(t-s)N^{-1}.
\end{equation}
We only need to apply Doob's inequality.
\end{proof}

\begin{proof}[Proof of Proposition \ref{prop:q}]
For $(\romannumeral1)$, notice that \eqref{eq:tight-3} and the weak convergence yield that for any fixed $\psi\in\mathcal C([0,1])$, $\Q\{\sup_{t\in[0,T]} |\langle \pi(t),\psi \rangle| \le C\|\psi\|_{L^1}\}=1$.
Since $\mathcal C([0,1])$ is separable, by standard density argument it holds $\Q$-almost surely that
\[\sup_{t\in[0,T]} \big|\langle \pi(t),\psi \rangle\big| \le C\|\psi\|_{L^1}, \quad \forall\,\psi\in\mathcal C([0,1]).\]
So $\pi(t)$ is absolutely continuous with respect to the Lebesgue measure on $[0,1]$ and can then be written as $\varpi(t,x)dx$.
Moreover, \eqref{eq:tight-2} assures that $\Q$ is concentrated on continuous paths, see \cite[Chapter 4, Remark 1.5]{klscaling}.
The continuity of $t \mapsto \varpi(t,\cdot)$ is then proved.

Since $(\romannumeral2)$ is a direct result from the proof of the tightness of $\{\nu^N\}$, we are left with $(\romannumeral3)$.
Pick $\varphi\in\mathcal C^1(\Sigma_T)$.
From the definition of $\nu_{t,x}^N$,
\[\int_0^1 \varphi(t,x)dx \left[ \int_0^1 \lambda\nu_{t,x}^N(d\lambda) \right] = \frac1N\sum_{i=K}^{N-K} \hat\eta_{i,K}(t)\bar\varphi_i(t),\]
where $\bar\varphi_i(t)$ is given by \eqref{eq:notion1}.
By the regularity of $\varphi$,
\[\left| \iint_{\Sigma_T} dxdt \left[ \varphi\int_0^1 \lambda\nu_{t,x}^N(d\lambda) \right] - \int_0^T \langle \pi^N,\varphi(t,\cdot) \rangle dt \right| \le \frac {C_\varphi TK}N.\]
Since $K = K(N) \ll N$, the right-hand side above vanishes as $N\to\infty$.
From the weak convergence, this yields that for any $\delta>0$ and fixed $\varphi\in\mathcal C^1(\Sigma_T)$,
\[\Q \left\{ \left| \iint_{\Sigma_T} dxdt \left[ \varphi\int_0^1 \lambda\nu_{t,x}(d\lambda) \right] - \int_0^T \langle \pi,\varphi(t,\cdot) \rangle dt \right| > \delta \right\} = 0.\]
By choosing a countable and dense subset of $\mathcal C^1(\Sigma_T)$ and applying $(\romannumeral1)$,
\[\Q \left\{ \iint_{\Sigma_T} \varphi \left[ \int_0^1 \lambda\nu_{t,x}(d\lambda) - \varpi \right] dxdt = 0, \ \forall\,\varphi\in\mathcal C^1(\Sigma_T) \right\} = 1.\]
The conclusion in $(\romannumeral3)$ then follows.
\end{proof}

\section{Compensated compactness}\label{sec:compensated compactness}

Given a Lax entropy--flux pair $(f,q)$, recall the entropy production defined in \eqref{eq:ent-prod}.
To simplify the notations, we will write $X(\nu,\varphi)$ when the choice of $(f,q)$ is clear.
Without loss of generality, we also fix $p=1$ to shorten the formulas.

This section is devoted to the proof of Proposition \ref{Q_dirac_property}.
Note that $\mathcal Y_d$, the subset of delta-type Young measures, is not closed in $\mathcal Y$, so $\Q_N(\mathcal Y_d)=1$ for all $N$ does not guarantee that $\Q(\mathcal Y_d)=1$.
From \cite[Proposition 2.1 \& Lemma 5.1]{Fritz04}, Proposition \ref{Q_dirac_property} follows from the next result, see also \cite[Section 5.6]{fritz2004derivation}.

\begin{proposition}\label{prop:compensated compactness}
Fix an arbitrary Lax entropy--flux pair $(f,q)$. Let $\varphi=\phi\psi$ with $\phi \in \mathcal{C}_c^\infty (\Sigma_T)$ and $\psi \in \mathcal C^\infty(\R^2)$. Then, we have the following decomposition: \[X (\nu^N,\varphi) = Y_N (\varphi) + Z_N (\varphi),\] and there exist random variables $A_{N,\phi}$, $B_{N,\phi}$ independent of $\psi$, such that
\begin{align}
	%\lim_{N \rightarrow \infty} \E_{\mu_N} [|Y_N (\psi)|] \leq A_N \|\varphi\|_{H^1},&\quad \limsup_{N \rightarrow \infty} A_N = 0; \label{condition_1}\\
	%\sup_{N \geq 1} \E_{\mu_N} [|Z_N (\psi)|] \leq  B_N  \|\varphi\|_{L^\infty},& \quad \sup_{N \geq 1} B_N < \infty. \label{condition_2}
	&|Y_N(\varphi)| \le A_{N,\phi}\|\psi\|_{H^1}, \quad \limsup_{N\to\infty} \E_{\mu_N} [A_{N,\phi}] = 0;\label{condition_1}\\
	&|Z_N(\varphi)| \le B_{N,\phi}\|\psi\|_{L^\infty}, \quad \sup_{N\ge1} \E_{\mu_N} [B_{N,\phi}] < \infty.\label{condition_2}
\end{align}
Here, $\|\cdot\|_{H^1}$ and $\|\cdot\|_{L^\infty}$ are the $H^1$- and $L^\infty$-norm computed on $\Sigma_T=(0,T)\times(0,1)$.
%In particular, \textcolor{blue}{by  \cite[Proposition 2.1]{Fritz04} and \cite[Section 5.6]{fritz2004derivation}}, the conclusion in Proposition \ref{Q_dirac_property} holds.
\end{proposition}

\begin{remark}
Proposition \ref{prop:compensated compactness} is a microscopic (stochastic) synthesis of the Murat--Tartar theory established in \cite{Murat78,Tartar79}.
The role of the function $\phi\in\mathcal C_c^\infty(\Sigma_T)$ is to localize the estimates in \eqref{condition_1} and \eqref{condition_2} away from the boundaries where the potential $V$ is unbounded.
\end{remark}

\subsection{Basic decomposition}
We first prove a basic decomposition for the microscopic entropy production.
Given a bounded function $\varphi=\varphi(t,x)$,
%and for $i=K,\ldots,N-K$, we shorten
%\[\varphi_i (t) = \varphi \big(t,\tfrac{i}{N} - \tfrac{1}{2N}\big), \quad \bar{\varphi}_i (t) = N \int_{\tfrac{i}{N} - \tfrac{1}{2N}}^{\tfrac{i}{N} + \tfrac{1}{2N}} \varphi (t,x) dx.\]
%For a real sequence $\{a_i\}$, we define
%\[\nabla a_i =a_{i+1} - a_i, \quad \nabla^* a_i = a_{i-1} - a_i, \quad \Delta a_i = a_{i+1} - 2 a_i + a_{i-1}.\]
%In the following, we shall frequently use the following summation by parts formula: for any real sequences $\{a_i\}, \{b_i\}$,
%\begin{equation}\label{summation by parts}
%\sum_{i=K}^{N-K} a_i \nabla b_i = \sum_{i=K}^{N-K} b_i \nabla^* a_i + a_{N-K} b_{N-K+1} - a_{K-1} b_K.
%\end{equation}
%Let $J_i = \eta_{i} (1-\eta_{i+1})$.
recall the notations $\varphi_i=\varphi_i(t)$ and $\bar\varphi_i=\bar\varphi_i(t)$ defined in \eqref{eq:notion1}.
Denote by $\hat\eta_i(t)=\hat\eta_{i,K}(t)$, $\hat J_i(t)=\hat J_{i,K}(t)$ and $\hat G_i(t)=\hat G_{i,K}(t)$ the smoothly weighted averaged averages introduced in \eqref{eq:block-average} of $\eta_i(t)$, $J_i(t)=\eta_i(t)(1-\eta_{i+1}(t))$ and $G_i(t)=V_i(\eta_i(t)-\rho_i)$.
We shall abbreviate them to $\hat\eta_i$, $\hat J_i$ and $\hat G_i$ when there is no confusion.

\begin{lemma}\label{lem:decomposition}
Fix a Lax entropy--flux pair $(f,q)$.
For $\varphi \in \mathcal C^1(\R^2)$ such that $\varphi(T,\cdot)=0$,
% be given as in Proposition \ref{prop:compensated compactness}. Then,
	\begin{equation}\label{entropy decomposition}
		X(\nu^N,\varphi) = \frac{1}{N} \sum_{i=K}^{N-K} f (\hat{\eta}_i (0)) \bar{\varphi}_i (0) + \Acal_N + \mathcal{S}_N + \mathcal{G}_N + \Mcal_N + \Ecal_N.
		%\sum_{\ell=1}^3 \Ecal_{N,\ell} (\varphi),
	\end{equation}
The terms in \eqref{entropy decomposition} are defined below.
$\mathcal A_N$, $\mathcal S_N$ and $\mathcal G_N$ are given by
\begin{align*}
	&\Acal_N=\Acal_N (\varphi) :=\int_0^T  \sum_{i=K}^{N-K} \bar{\varphi}_i(t) f^\prime (\hat{\eta}_{i})  \nabla^* \left[ \hat{J}_{i} - J (\hat{\eta}_{i}) \right] dt,\\
	&\mathcal{S}_N=\mathcal{S}_N (\varphi) := \sigma_N \int_0^T \sum_{i=K}^{N-K} \bar{\varphi}_i(t) f^\prime (\hat{\eta}_{i}) \Delta \hat{\eta}_{i}   dt,\\
	&\mathcal{G}_N=\mathcal{G}_N (\varphi) := -\int_0^T \frac{1}{N} \sum_{i=K}^{N-K} \bar\varphi_i(t) f' (\hat{\eta}_{i}) \hat G_i\,dt.
\end{align*}
%where $\hat\eta_i=\hat\eta_{i,K}$, $\hat J_i=\hat J_{i,K}$ and $\hat G_i=\hat G_{i,K}$ are respectively the smoothly weighted averages introduced in \eqref{eq:block-average} of $\eta_i$, $J_i=\eta_i(1-\eta_{i+1})$ and $G_i=V_i(\eta_i-\rho_i)$.
$\Mcal_N$ is a martingale given by
\[\Mcal_N=\Mcal_N (\varphi) := -\int_0^T \frac{1}{N} \sum_{i=K}^{N-K} \bar{\varphi}'_i(t)M_i(t)dt,\]
where $M_i=M_i(t)$ is the Dynkin's martingale associated with $f(\hat\eta_i(t))$, see \eqref{decom_1} below.
Finally, $\Ecal_N=\Ecal_{N,1}+\Ecal_{N,2}+\Ecal_{N,3}$ is defined through
\begin{align*}
	\Ecal_{N,1}=\Ecal_{N,1} (\varphi) &:=  -f(0)\iint_{B_{T,N}} \partial_t\varphi\,dxdt - q(0)\iint_{B_{T,N}} \partial_x\varphi\,dxdt,\\
	\Ecal_{N,2}=\Ecal_{N,2} (\varphi) &:= \int_0^T \sum_{i=K}^{N-K}  \big[ \bar{\varphi}_i  \nabla^* q (\hat{\eta}_i) - q (\hat{\eta}_i)  \nabla \varphi_i  \big] dt,\\
	\Ecal_{N,3}=\Ecal_{N,3} (\varphi) &:= \int_0^T  \frac1N\sum_{i=K}^{N-K} \bar{\varphi}_i \left[\epsilon_{i,K}^{(1)}  + \epsilon_{i,K}^{(2)} \right] dt,
%\Ecal_{N,3} (\varphi) &:=  - q (0) \int_0^T \Big(\int_0^{\tfrac{2K-1}{2N}} + \int_{1-\tfrac{2K-1}{2N}}^{1}\Big) \partial_x \varphi dx\,dt \\
%&\qquad - f (0) \int_0^T \Big(\int_0^{\tfrac{2K-1}{2N}} + \int_{1-\tfrac{2K-1}{2N}}^{1}\Big) \partial_t \varphi dx\,dt,
\end{align*}
where $B_{T,N}$ is the region $t\in[0,T]$ and $x\in[0,\tfrac{2K-1}{2N})\cup[1-\tfrac{2K-1}{2N},1]$,
\[\epsilon_{i,K}^{(1)}  = L_N f (\hat{\eta}_i)  - f^\prime (\hat{\eta}_i) L_N \hat{\eta}_i, \quad \epsilon_{i,K}^{(2)} = N\big[f^\prime (\hat{\eta}_i)\nabla^* J(\hat{\eta}_i) - \nabla^* q (\hat{\eta}_i)\big].\]
\end{lemma}

%\begin{remark}
%Since $\varphi = \varphi \phi$ has compact support in $[0,T) \times (0,1)$ and $K \ll N$,   the error term $\Ecal_{N,3} (\varphi)$ is identically zero for $N$ large enough.
%\end{remark}

\begin{proof}
By the definition of $\nu^N$, 
\begin{equation}\label{decom_2}
	X(\nu^N,\varphi) =  - \int_0^T \frac{1}{N} \sum_{i=K}^{N-K} f (\hat{\eta}_i) \bar{\varphi}_i^\prime\,dt - \int_0^T \sum_{i=K}^{N-K} q (\hat{\eta}_i)\nabla \varphi_i\,dt + \Ecal_{N,1} (\varphi).
\end{equation}
By Dynkin's formula, for $K \leq i \leq N-K$,
\begin{equation}\label{decom_1}
	M_i (t) := f (\hat{\eta}_i (t)) - f (\hat{\eta}_i (0)) - \int_0^t  L_N \big[f (\hat{\eta}_i (s))\big] ds, \quad t\in[0,T],
\end{equation}
defines a martingale.
Since $\varphi$ vanishes at $t=T$, $\Mcal_N(\varphi)$ satisfies that
\begin{equation}\label{decom_3}
	\begin{aligned}
	- \int_0^T \frac1N\sum_{i=K}^{N-K} f(\hat\eta_i)\bar\varphi'_i\,dt =\,&\frac1N\sum_{i=K}^{N-K} f(\hat\eta_i(0))\bar\varphi_i(0)\\
	&+ \int_0^T \frac1N\sum_{i=K}^{N-K} \bar\varphi_i L_N \big[f(\hat\eta_i)\big]dt + \Mcal_N(\varphi).
	\end{aligned}
\end{equation}
Recall the definition of $L_N$ in \eqref{eq:generator}.
Note that, for $K \leq i \leq N-K$, $\hat{\eta}_i$ does not depend on $\eta_{0}$ or $\eta_{N}$, and thus $L_{\rm bd} \hat{\eta}_i = L_{\rm bd} \hat{\eta}_i  = 0$.
Notice that
\[L_N [\hat{\eta}_i] = N \nabla^* \hat{J}_i + N\sigma_N \Delta \hat{\eta}_i %+ \widehat{(V\rho)}_i - \widehat{(V\eta)}_i.\]
- \hat G_i, \quad \forall\,i=K,\ldots,N-K.\]
Therefore, for $i=K$, ...,  $N-K$, $L_N[f(\hat\eta_i)]$ is equal to
\begin{align*}
	\epsilon_{i,K}^{(1)} + f^\prime (\hat{\eta}_i) &\left( N \nabla^* \hat{J}_i + N\sigma_N \Delta \hat{\eta}_i - \hat G_i \right) = \epsilon_{i,K}^{(1)} + \epsilon_{i,K}^{(2)}\\
	&+\,Nf^\prime (\hat{\eta}_i) \nabla^* \left[ \hat{J}_i-J(\hat\eta_i) \right] + N \nabla^* q(\hat\eta_i) + f^\prime (\hat{\eta}_i) \left( N\sigma_N \Delta \hat{\eta}_i - \hat G_i \right).
\end{align*}
From the above formula of $L_N[f(\hat\eta_i)]$,
\begin{equation}\label{decom_4}
	\begin{aligned}
	\int_0^T \frac1N\sum_{i=K}^{N-K} \bar\varphi_i L_N \big[f(\hat\eta_i)\big]dt = &\int_0^T \sum_{i=K}^{N-K} \bar\varphi_i\nabla^*q(\hat\eta_i)\,dt \\
	&+\,\Acal_N(\varphi) + \mathcal S_N(\varphi) + \mathcal G_N(\varphi) + \Ecal_{N,3}(\varphi).
	\end{aligned}
\end{equation}
We then conclude the proof by inserting \eqref{decom_3} and \eqref{decom_4} into the first term on the right-hand side of \eqref{decom_2}.
%By writing
%\[\bar{\varphi}_i f^\prime (\hat{\eta}_i) \nabla^* \hat{J}_i  = \bar{\varphi}_i f^\prime (\hat{\eta}_i) \nabla^*[ \hat{J}_i - J(\hat{\eta}_i)] + \bar{\varphi}_i [f^\prime (\hat{\eta}_i) \nabla^* J(\hat{\eta}_i) - \nabla^* q (\hat{\eta}_i)] +  \bar{\varphi}_i  \nabla^* q (\hat{\eta}_i),\]
%we conclude the proof.
\end{proof}

\subsection{Dirichlet forms}

Given a function $\alpha:[0,1]\to(0,1)$, denote by $\nu^N_{\alpha(\cdot)}$ the product measure on $\Omega_N$ with marginals
\[\nu^N_{\alpha(\cdot)} (\eta(i)=1) = \alpha(i/N), \quad \forall\,i = 0,1, \ldots,N.\]
When $\alpha(\cdot) \equiv \alpha$ is a constant, we shorten the notation as $\nu^N_{\alpha(\cdot)} = \nu^N_{\alpha}$.
Given two probability measures $\nu$ and $\mu$ on $\Omega_N$, let $f:=\mu/\nu$ be the density function.
Define
\begin{align}
	\label{eq:dirichlet_exc}
	D_{\rm exc}^N (\mu;\nu)	&:= \frac{1}{2} \sum_{\eta \in \Omega_N} \sum_{i=0}^{N-1} \Big( \sqrt{f(\eta^{i,i+1})} - \sqrt{f(\eta)} \Big)^2 \nu(\eta),\\
	\label{eq:dirichlet_g}
	D_{\rm G}^N (\mu;\nu) &:= \frac{1}{2N} \sum_{\eta \in \Omega_N} \sum_{i=1}^{N-1} c_{i,{\rm G}} (\eta) \Big( \sqrt{f(\eta^{i})} - \sqrt{f(\eta)} \Big)^2 \nu(\eta),\\
	\label{eq:dirichlet_left}
	D_-^N (\mu;\nu) &:= \frac{1}{2} \sum_{\eta \in \Omega_N} c_0 (\eta) \Big( \sqrt{f(\eta^{0})} - \sqrt{f(\eta)} \Big)^2 \nu(\eta),\\
	\label{eq:dirichlet_right}
	D_+^N (\mu;\nu) &:= \frac{1}{2} \sum_{\eta \in \Omega_N} c_N (\eta) \Big( \sqrt{f(\eta^{N})} - \sqrt{f(\eta)} \Big)^2 \nu(\eta),
\end{align}
with $c_{i,{\rm G}}$ in \eqref{eq:c_iG} and $c_0$, $c_N$ in \eqref{eq:c_bd}.
Note that $D_-^N\equiv0$ if $c_{\rm in}^-=c_{\rm out}^-=0$, and similarly for $D_+^N$.
Let $\mu_t^N$ be the distribution of the process at time $t$.
Define
\[D_{\rm exc}^N(t) = D_{\rm exc}^N (\mu_t^N;\nu_{\frac12}^N), \quad D_{\rm G}^N(t) = D_{\rm G}^N(\mu_t^N;\nu_{\rho(\cdot)}^N),
%, \quad D_\pm^N(t) = D_\pm(\mu_t^N;\nu_{u_\pm}^N),
\quad \forall\,t\ge0.\]
Let $c_-=c_{\rm in}^-(c_{\rm in}^-+c_{\rm out}^-)^{-1}$ when at least one of $c_{\rm in}^-$ and $c_{\rm out}^-$ is positive.
Define
\[D_-^N(t) = D_-^N (\mu_t^N;\nu_{c_-}^N), \quad \forall\,t\ge0.\]
If $c_{\rm in}^-=c_{\rm out}^-=0$ we fix $D_-^N(t)\equiv0$.
Let $c_+$ and $D_+^N(t)$ be defined similarly.

\begin{lemma}\label{lem:dirichlet_form_estimate}
%Suppose that at least one of the following conditions holds: (1) $V \in L^1((0,1))$, (2) $a\le0$, (3) $u_-=\rho(0)$, $u_+=\rho(1)$.
%Then,
For any $t>0$, there is a constant $C$ independent of $N$, such that
\[\int_0^t \Big[\sigma_ND_{\rm exc}^N(s) + D_{\rm G}^N(s) + D_-^N(s) + D_+^N(s)\Big]\,ds \le C.\]
\end{lemma}

\begin{proof}
Let $\nu=\nu_\alpha^N$ with $\alpha\equiv\tfrac12$.
For a probability measure $\mu=f\nu$ on $\Omega_N$, from the calculation in Appendix \ref{sec:df_estimates},
	\begin{align}
		\big\langle f, L_{\rm exc}[\log f] \big\rangle_\nu &\le -(\sigma_N+1)D_{\rm exc}^N(\mu;\nu) + C,\label{df_1}\\
		\big\langle f,L_{\rm G}[\log f] \big\rangle_\nu &\le -2D_{\rm G}^N(\mu;\nu_{\rho(\cdot)}^N) + \frac{1}{N}\sum_{i=1}^{N-1} V_i \log \left( \tfrac{\rho_i}{1-\rho_i} \right) (\rho_i-E^\mu [\eta_i]),\label{df_2}\\
		\big\langle f, L_{\rm bd}[\log f] \big\rangle_\nu &\le -2\big[D_-^N(\mu;\nu_{c_-}^N) + D_+^N(\mu;\nu_{c_+}^N)\big] + C.\label{df_3}%\\
		%+ (c_{\rm in}^--c_{\rm out}^-)(1-2E^\mu [\eta_0]) + (c_{\rm in}^+-c_{\rm out}^+)(1-2E^\mu [\eta_N]),
	\end{align}
%where $\alpha_i=\alpha(i/N)$ and $\psi(\alpha)=\log \alpha-\log (1-\alpha)$.
	For $s\in[0,t]$, let $f_s^N:=\mu_s^N/\nu$, then% and $\alpha\equiv\tfrac12$,
	\begin{align*}
		\big\langle f_s^N,L_N \log f_s^N \big\rangle_\nu \le\,&- 2N\sigma_ND^N_{\rm exc}(s) - 2ND_{\rm G}^N(s) - 2N\big[D_-^N(s)+D_+^N(s)\big]\\
		&+ \sum_{i=1}^{N-1} V_i \log \left( \tfrac{\rho_i}{1-\rho_i} \right) \big(\rho_i-\E_{\mu_N} [\eta_i(s)]\big) + CN.
	\end{align*}
	Applying Lemma \ref{lem:auxiliary} below, we obtain the estimate
	\[\int_0^t \Big[\sigma_ND_{\rm exc}^N(s) + D_{\rm G}^N(s) + D_-^N(s) + D_+^N(s)\Big]\,ds \le  \int_0^t \frac{\langle f_s^N,-L_N [\log f_s^N] \rangle_\nu}{2N}\,ds + C.\]
	Standard manipulation gives that
	\[\int_0^t \big\langle f_s^N,-L_N [\log f_s^N] \big\rangle_\nu\,ds = \sum_{\eta\in\Omega_N} \log[f_0^N(\eta)]\mu_0^N(\eta) - \sum_{\eta\in\Omega_N} \log[f_t^N(\eta)]\mu_t^N(\eta)\]
	is bounded by $C'N$, so we conclude the proof.
\end{proof}

The following a priori bound is used in the previous proof.

\begin{lemma}\label{lem:auxiliary}
Suppose that $a\in\mathcal C^1([0,1])$ has Lipschitz continuous derivative.
Let $a_i=a_i^N=a(\tfrac iN)$.
Then, there is a constant $C$ independent of $N$, such that
\[\E_{\mu_N} \left[ \int_0^t \frac1N\sum_{i=1}^{N-1} V_ia_i\big(\rho_i-\eta_i(s)\big)\,ds \right] \le C.\]
In particular, the profile $a(x)=\log[\rho(x)]-\log[1-\rho(x)]$ satisfies the condition in the lemma.
\end{lemma}

\begin{proof}
	By Dynkin's formula, 
	\[\E_{\mu_N} \left[ \frac1N\sum_{i=0}^N a_i \big(\eta_{i}(t) - \eta_{i}(0)\big) - \int_0^t \frac1N\sum_{i=0}^N a_i L_N \eta_{i}(s)\,ds \right] = 0.\]
	From the definition of $L_N$,
	\begin{align*}
		\frac1N\sum_{i=0}^N a_iL_N \eta_{i} =\,&\frac1N\sum_{i=1}^{N-1} V_ia_i(\rho_i - \eta_i) + \sum_{i=0}^{N-1} (\nabla a_i)j_{i,i+1}\\
		&+ a_0\big[c_{\rm in}^- - (c_{\rm in}^-+c_{\rm out}^-)\eta_0\big] + a_N\big[c_{\rm in}^+ - (c_{\rm in}^++c_{\rm out}^+)\eta_N\big],
	\end{align*}
	where $j_{i,i+1}=\eta_i(1-\eta_{i+1})+\sigma_N(\eta_i-\eta_{i+1})$.
	As $|a_i| \le |a|_\infty$ and $|\eta_i| \le 1$,
	\[\E_{\mu_N} \left[ \int_0^t \frac1N\sum_{i=1}^{N-1} V_ia_i(\rho_i - \eta_i)\,ds + \int_0^t \sum_{i=0}^{N-1} (\nabla a_i)j_{i,i+1}\,ds  \right] \le C.\]
	Furthermore, using sum-by-parts formula,
	\begin{align*}
		\sum_{i=0}^{N-1} (\nabla a_i)j_{i,i+1} =\,&\sum_{i=0}^{N-1} (\nabla a_i)\eta_i(1-\eta_{i+1})\\
		&+ \sigma_N \left[ \sum_{i=1}^{N-1} (\Delta a_i)\eta_i  + (\nabla a_0)\eta_0 - (\nabla a_{N-1})\eta_N \right].
	\end{align*}
	Since $\sigma_N = o(N)$, one can conclude from the regularity of $a(\cdot)$.
\end{proof}

\subsection{Proof of Proposition \ref{prop:compensated compactness}} 

We fix $\phi\in\mathcal C_c^\infty(\Sigma_T)$ and prove Proposition \ref{prop:compensated compactness} by estimating each term in the decomposition \eqref{entropy decomposition} uniformly in $\psi$.
Compared to the proofs of \cite[Proposition 6.1]{xu2021hydrodynamic} and \cite[Lemma 4.3]{xu2022hydrodynamics}, extra effort is needed to take care of the term with respect to $V$.
In the following contents, $C$ is constant that may depend on $(f,q)$ and $T$ but is independent of $(\phi,\psi)$, while $C_\phi$ is constant that also depends on $\phi$ but is independent of $\psi$.

We first show that the error term $\Ecal_N$ vanishes in the limit.

\begin{lemma}\label{lem: error term}
$\Ecal_{N,1}\equiv0$ for sufficiently large $N$.
Moreover, as $N \rightarrow \infty$,
\begin{enumerate}
	\item[$(\romannumeral1)$] $\Ecal_{N,2}$ vanishes uniformly in $H^{-1} (\Sigma_{T})$, thus satisfies \eqref{condition_1}; 
	\item[$(\romannumeral2)$] $\Ecal_{N,3}$  vanishes uniformly in $\Mcal (\Sigma_T)$, thus satisfies \eqref{condition_2}. 
\end{enumerate}
\end{lemma}

\begin{proof}
For $\phi\in\mathcal C_c^1(\Sigma_T)$, there is $\delta_\phi>0$ such that $\phi(t,x)=0$ for $x\notin(\delta_\phi,1-\delta_\phi)$.
Note that $K \ll N$, so we can find $N_\phi$ depending only on $\phi$ such that $\varphi=\phi\psi=0$ on $B_{T,N}$ for all $\psi$ and $N>N_\phi$.
The vanishment of $\Ecal_{N,1}$ then follows.

We first prove $(\romannumeral1)$.
For $N>N_\phi$, we can perform summation by parts without generating boundary term:
\[	\Ecal_{N,2} (\psi) = \int_0^T \sum_{i=K}^{N-K}  \big( \bar{\varphi}_i - \varphi_i  \big) \nabla^* q (\hat{\eta}_i)dt.% + \int_{0}^T \big[q(\hat{\eta}_{K-1}) \psi_K - q(\hat{\eta}_{N-K}) \psi_{N-K+1}\big] dt.
\]
Note that for some $\theta_x \in ((2i-1)/2N,x)$,
\[	|\bar{\varphi}_i - \varphi_i| \leq \int_{\tfrac{i}{N} - \frac{1}{2N}}^{{\tfrac{i}{N} + \frac{1}{2N}}} |\partial_x\varphi(\theta_x)| dx,
\quad |\nabla^* q (\hat{\eta}_i)| \le C\big|\hat\eta_{i-1}-\hat\eta_i\big| \le CK^{-1}.\]
%, and since $\psi$ has compact support,
%\begin{equation}\label{compact_support_inequ}
%\psi(t,y) = \int_0^y \partial_x \psi (t,x) dx \leq \sqrt{y} \Big(\int_0^y (\partial_x \psi (t,x))^2 dx \Big)^{1/2}.
%\end{equation}
Thus, by Cauchy-Schwarz inequality and that $\varphi=\psi\phi$,
\[|\Ecal_{N,2} (\varphi) | %\lesssim (K^{-1} + \sqrt{K/N}) \|\psi\|_{H^{1}}  \lesssim (K^{-1} + \sqrt{K/N}) \|\varphi\|_{H^{1}}.\]
\le CK^{-1}\|\varphi\|_{H^1} \le C_\phi K^{-1}\|\psi\|_{H^1}.\]
%Above, the proportional constant may depend on $\phi, q, T$. This proves the first statement. 

To prove $(\romannumeral2)$, Taylor's expansion gives that% and using $f'J'=q'$,
\[\left| \epsilon_{i,K}^{(1)} \right| \le C\sum_{j=i-K}^{i+K} \left[ N \sigma_N (\hat{\eta}^{j,j+1}_i - \hat{\eta}_i)^2 + V \left( \tfrac jN \right) (\hat{\eta}^j_i - \hat{\eta}_i)^2 \right].\]
By the definition of $\hat\eta_i$ in \eqref{eq:block-average}, $|\hat\eta_i^{j,j+1}-\hat\eta_i| = K^{-2}$ and $|\hat\eta_i^j-\hat\eta_i| = w_{j-i}$, so
\[%\lesssim \sum_{|j-i| \leq K} V_j w_{j-i}^2 + \frac{N\sigma_N}{K^3},
	\left| \varepsilon_{i,K}^{(1)} \right| \le CN\sigma_NK^{-3} + C\sum_{|j|<K} V \left( \tfrac{i+j}N \right) w_j^2.\]
As $\varphi=\phi\psi$, the definition \eqref{eq:notion1} of $\bar\varphi_i$ yields that
\[\frac1N\sum_{i=K}^{N-K} \bar\varphi_i \sum_{|j|<K} V \left( \tfrac{i+j}N \right) w_j^2 \le \|\psi\|_{L^\infty}\sum_{|j|<K} w_j^2 \sum_{i=K}^{N-K} V \left( \tfrac{i+j}N \right) \int_{\frac iN-\frac1{2N}}^{\frac iN+\frac1{2N}} |\phi|\,dx.\]
Recall that $\phi=0$ for $x\notin(\delta_\phi,1-\delta_\phi)$.
Let $N'_\phi$ be such that $KN^{-1}<2^{-1}\delta_\phi$ for $N>N'_\phi$.
Then, for every $N>N'_\phi$, the expression above is bounded by
\[\|\psi\|_{L^\infty}\sum_{|j|<K} w_j^2 \left( \|\phi V\|_{L^1} + \|\phi\|_{L^1}\sup \left\{ |V'|; x\in \left[\frac{\delta_\phi}2,1-\frac{\delta_\phi}2 \right] \right\} \frac{2K+1}{2N} \right).\]
From the choice of $w_j$ in \eqref{eq:block-average}, it is bounded by $C_\phi K^{-1}\|\psi\|_{L^\infty}$.
Similarly, using Taylor's expansion and the relation $f'J'=q'$,
\[\left| \epsilon_{i,K}^{(2)} \right| \le CN(\hat{\eta}_{i-1} - \hat{\eta}_i)^2 \le CNK^{-2}.\]
Putting the estimates above together,
\[|\Ecal_{N,3} (\psi) | \le C_\phi \left( \frac{N\sigma_N}{K^3} + \frac N{K^2} + \frac1K \right) \|\psi\|_{L^\infty}.\]
The proof is then concluded by the choice of $K$ in \eqref{eq:mesoscopic}.
\end{proof}

Similarly, with the compactness of $\phi$ we can carry out the estimate for $\mathcal{G}_N$.

\begin{lemma}
	The functional $\mathcal{G}_N$ satisfies \eqref{condition_2}.
\end{lemma}

\begin{proof}
	Recall that $G_i=V_i(\eta_i-\rho_i)$, so $|G_i| \le C|V(\tfrac iN)|$.
	Then,
	\[|\mathcal{G}_N (\varphi)| \le C\int_0^T \frac1N\sum_{i=K}^{N-K} |\bar\varphi_i|\sum_{|j|<K} \left| V \left( \tfrac{i+j}N \right) \right| w_j\,dt.\]% \leq C T \|f^\prime\|_{L^\infty} \|\phi V\|_{L^1} \|\varphi\|_{L^\infty},\]
	By the argument in Lemma \ref{lem: error term} $(\romannumeral2)$, it is bounded by $C_\phi\|\psi\|_{L^\infty}$ uniformly in $N$.
\end{proof}

Now, we deal with the martingale term.

\begin{lemma}\label{lem:martingale}
The martingale $\Mcal_N (\psi)$ satisfies \eqref{condition_1}.
\end{lemma}

\begin{proof}
Since $\partial_t\varphi=\phi\partial_t\psi+\psi\partial_t\phi$, using \CS inequality,
\[|\Mcal_N(\varphi)|^2 \le \|\psi\|_{H^1}^2 \int_0^T \sum_{i=K}^{N-K} \int_{\frac iN-\frac1{2N}}^{\frac iN+\frac1{2N}} \big[\phi^2+(\partial_t\phi)^2\big]M_i^2(t)dxdt.\]
Through Dynkin's formula, $M_i$ satisfies that
\begin{align*}
	\E_{\mu_N} \big[M_i^2(t)\big] &= \E_{\mu_N} \left[ \int_0^t \sum_{j=i-K}^{i+K} N\big(\eta_j(1-\eta_{j+1})+\sigma_N\big) \big(f(\hat\eta_i^{j,j+1})-f(\hat\eta_i)\big)^2ds \right.\\
	&+ \left. \int_0^t \sum_{j=i-K+1}^{i+K-1} V_j\big[\eta_j(1-\rho_j)+\rho_j(1-\eta_j)\big] \big(f(\hat\eta_i^j)-f(\hat\eta_i)\big)^2ds \right]\\
	&\le CN\sigma_N\sum_{j=i-K}^{i+K} \E_{\mu_N} \big[(\hat\eta_i^{j,j+1}-\hat\eta_i)^2 \big] + C\sum_{|j|<K} V \left( \tfrac{i+j}N \right) w_j^2.
\end{align*}
Similarly to the estimate in Lemma \ref{lem: error term}$(\romannumeral2)$, $|\Mcal_N(\varphi)| \le a_{N,\phi}\|\psi\|_{H^1}$ with
\[\E_{\mu_N} [a_{N,\phi}] \le C_\phi \sqrt{\frac{N\sigma_N}{K^3} + \frac1K}.\]
The proof is then concluded by the choice of $K$ in \eqref{eq:mesoscopic}.
\end{proof}

To deal with $\Acal_N$ and $\mathcal S_N$, we need the block estimates stated below.
They follow from the upper bound of $D_{\rm exc}^N$ in Lemma \ref{lem:dirichlet_form_estimate} and the logarithmic Sobolev inequality for exclusion process \cite{Yau97}.
The proofs are the same as \cite[Proposition 6.4 \& 6.5]{xu2021hydrodynamic} and \cite[Proposition 4.6 \& 4.7]{xu2022hydrodynamics}.
For this reason, we omit the details here.

\begin{lemma}\label{lem:block_estimate}
	There exists some finite constant $C$ independent of $N$, such that 
	\begin{align}
		&\E_{\mu_N} \Big[\int_{0}^{T} \sum_{i=K}^{N-K} \Big(\hat{J}_i (t) - J (\hat{\eta}_i (t))\Big)^2 dt\Big] \leq C \Big(\frac{K^2}{\sigma_N} + \frac{N}{K}\Big),\label{eq:one-block}\\
		&\E_{\mu_N} \Big[\int_{0}^{T} \sum_{i=K}^{N-K} \Big(\nabla \hat{\eta}_i (t)\Big)^2 dt\Big] \leq C \Big(\frac{1}{\sigma_N} + \frac{N}{K^3}\Big).\label{eq:h1}
	\end{align}
\end{lemma}

Using Lemma \ref{lem:block_estimate}, we can conclude the decompositions of $\Acal_N$ and $\mathcal S_N$.

\begin{lemma}\label{lem:A_NS_N}
Define functionals $\Acal_{N,1}$ and $\mathcal S_{N,1}$ respectively by
\begin{align}
	&\Acal_{N,1}(\varphi) := \int_0^T \sum_{i=K}^{N-K} \bar\varphi_i \left[ \hat{J}_{i} - J (\hat{\eta}_{i}) \right] \nabla f'(\hat\eta_i)dt,\\
	&\mathcal S_{N,1}(\varphi) := -\,\sigma_N\int_0^T \sum_{i=K}^{N-K} \bar\varphi_i\nabla\hat\eta_i\nabla f'(\hat\eta_i)dt.\label{eq:S_N1}
\end{align}
Then, $\Acal_N-\Acal_{N,1}$ and $\mathcal{S}_N-\mathcal S_{N,1}$ satisfy  \eqref{condition_1}, while $\Acal_{N,1}$ and $\mathcal S_{N,1}$ satisfy \eqref{condition_2}.
\end{lemma}

The proof of Lemma \ref{lem:A_NS_N} follows \cite[Lemma 6.6 \& 6.7]{xu2021hydrodynamic} almost line by line, so we only sketch the difference.
It is worth noting that, $\mathcal S_{N,1}$ turns out to be the only term that survives in the limit, eventually generates the non-zero macroscopic entropy in \eqref{entropy_inequality}.

\begin{proof}

We first treat $\Acal_N$.
Since $\varphi=\phi\psi$ with $\phi\in\mathcal C_c^\infty (\Sigma_T)$ being fixed, similarly to the proof of Lemma \ref{lem: error term} $(\romannumeral1)$, we have for $N>N_\phi$ that
\[(\Acal_N-\Acal_{N,1})(\varphi) = \int_0^T \sum_{i=K}^{N-K} f'(\hat\eta_{i+1}) \left[ \hat{J}_{i} - J (\hat{\eta}_{i}) \right] \nabla\bar\varphi_i\,dt.\]
Applying \CS inequality and Lemma \ref{lem:block_estimate}, we obtain that $|(\Acal_N-\Acal_{N,1})(\varphi)| \le a_N\|\varphi\|_{H^1}$ and $|\Acal_{N,1}(\varphi)| \le b_N\|\varphi\|_{L^\infty}$, where $(a_N,b_N)$ are random variables such that
\[\E_{\mu_N} [a_N] \le C \sqrt{\frac{K^2}{N\sigma_N} + \frac1K}, \quad \E_{\mu_N} [b_N] \le C \left( \frac K{\sigma_N} + \frac N{K^2} \right).\]
Noting that $\|\varphi\|_{H^1} \le C_\phi\|\psi\|_{H^1}$, $\|\varphi\|_{L^\infty} \le C_\phi\|\psi\|_{L^\infty}$, the conclusion follows from \eqref{eq:mesoscopic}.

The proof for $\mathcal S_N$ is similar.
For $N>N_\phi$,
\[(\mathcal S_N-\mathcal S_{N,1})(\varphi) = -\sigma_N\int_0^T \sum_{i=K}^{N-K} f'(\hat\eta_{i+1})\nabla\bar\varphi_i\nabla\hat\eta_i\,dt.\]
By \CS inequality and Lemma \ref{lem:block_estimate}, $|(\mathcal S_N-\mathcal S_{N,1})(\varphi)| \le a'_N\|\varphi\|_{H^1}$ and $|\mathcal S_{N,1}(\varphi)| \le b'_N\|\varphi\|_{L^\infty}$ with random variables $(a'_N,b'_N)$ satisfying
\[\E_{\mu_N} [a'_N] \le C\sqrt{\frac{\sigma_N}N+\frac{\sigma_N^2}{K^3}}, \quad \E_{\mu_N} [b'_N] \le C \left( 1+\frac{N\sigma_N}{K^3} \right).\]
The conclusion follows similarly.
\end{proof}

\section{Measure-valued entropy solution}
\label{sec:mve}

We prove that under $\Q$, $\nu$ satisfies %both \eqref{eq:micro-energy-bd} and
\eqref{eq:micro-entropy-inequality} with probability $1$.
%By Proposition \ref{prop:compensated compactness}, any weak limit $\Qcal$ of $\Qcal^N$
We call such a Young measure %that satisfies \eqref{eq:micro-energy-bd} and \eqref{eq:micro-entropy-inequality}
a \emph{measure-valued entropy solution} to the initial-boundary value problem \eqref{hyperbolic_pde} and \eqref{eq:bl-bd}.
Thanks to Proposition \ref{Q_dirac_property}, $\nu$ is essentially
%on trajectories  such that for $(t,x)$-a.e., $\nu_{t,x} (du) = \delta_{u (t,x)} (du)$ for some $u \in L^\infty (\Sigma_T)$.
the entropy solution.

\subsection{Proof of \eqref{eq:micro-entropy-inequality}}
\label{subsec:micro-ent-ineq}

Fix an arbitrary Lax entropy--flux pair $(f,q)$ and $\varphi\in\mathcal C_c^2([0,T)\times(0,1))$ such that $\varphi\ge0$.
As in Section \ref{sec:compensated compactness}, we shall let $N\to\infty$ and examine the limit of each term in the decomposition \eqref{entropy decomposition}.
The main difference is that the proof of Proposition \ref{prop:compensated compactness} requires uniform estimate in the test function, which is no more necessary here.

First, by Lemma \ref{lem: error term}, \ref{lem:martingale} and \ref{lem:A_NS_N}, for any $\delta>0$,
\[\lim_{N\to\infty} \Q_N \Big\{ |\Acal_N(\varphi)| + |\mathcal S_N(\varphi)-\mathcal S_{N,1}(\varphi)| + |\Mcal_N(\varphi)| + |\Ecal_N(\varphi)| > \delta \Big\} = 0.\]
Meanwhile, the convexity of $f$ ensures that $\nabla\hat\eta_i\nabla f'(\hat\eta_i) \ge 0$, so $\mathcal S_{N,1}(\varphi)$ given by \eqref{eq:S_N1} is non-positive everywhere.
Therefore, \eqref{eq:micro-entropy-inequality} holds $\Q$-almost surely if we can show that
\begin{equation}\label{eq:convergence-g}
	\limsup_{N\to\infty} \Q_N \left\{ \left| \mathcal G_N(\varphi) + \iint_{\Sigma_T} \varphi(t,x)\int_\R f'(\lambda)G(x,\lambda)\,\nu_{t,x}^N(d\lambda)\,dxdt \right| > \delta \right\} = 0,
\end{equation}
for any $\delta>0$, where $G(x,\lambda)=V(x)(\lambda-\rho(x))$.
Recall the definition of $\mathcal G_N(\varphi)$ in Lemma \ref{lem:decomposition}.
As $\varphi$ is compactly supported, for sufficiently large $N$ we have
\begin{align*}
	\mathcal G_N(\varphi) =\,&- \iint_{\Sigma_T} \varphi(t,x)\int_\R f'(\lambda)G(x,\lambda)\,\nu_{t,x}^N(d\lambda)\,dxdt\\
	&- \int_0^T \sum_{i=K}^{N-K} f'(\hat\eta_i)\int_{\frac iN-\frac1{2N}}^{\frac iN+\frac1{2N}} \varphi \big[\hat G_i-G(\cdot,\hat\eta_i)\big]\,dxdt,
\end{align*}
where $\hat G_i$ is the smoothly weighted average of $G_i=V(\tfrac iN)(\eta_i-\rho(\tfrac iN))$.
Straightforward computation shows that
\[\left| \int_0^T \sum_{i=K}^{N-K} f'(\hat\eta_i)\int_{\frac iN-\frac1{2N}}^{\frac iN+\frac1{2N}} \varphi \big[\hat G_i-G(\cdot,\hat\eta_i)\big]\,dxdt \right| \le \frac{CK}N,\]
which vanishes uniformly as $N\to\infty$.
We can then conclude \eqref{eq:convergence-g}.

\subsection{Direct proof of \eqref{eq:micro-energy-bd}}
\label{subsec:micro-energy-bd}

%To conclude the proof of Theorem \ref{thm:hydronamics} in the non-integrable case, it remains to prove the following energy estimate.
%
%\begin{proposition}\label{prop:energy bound}
%Assume $V \notin L^1 ((0,1))$. Then, $\mathcal{Q}$-a.s., $u (t,x)$ satisfies the energy bound \eqref{energy bound}.
%\end{proposition}
%
%The proof of the above result relies on the following lemma.
%
Notice that \eqref{eq:micro-energy-bd} holds if for some constant $C_0$,
\begin{equation}\label{eq:energy-bd-1}
	\Q \left\{ \sup_{g\in\mathcal C_c^\infty(\R\times(0,1))} \left\{\iint_{\Sigma_T} (\bar u-\rho)gV\,dxdt - C_0\iint_{\Sigma_T} g^2V\,dxdt \right\} < \infty \right\} = 1,
\end{equation}
where $\bar u=\bar u(t,x):=\int_\R \lambda\nu_{t,x}(d\lambda)$ is a measurable function on $\Sigma_T$.

Let $\{g^j;j \geq 1\}$ be a countable subset of $\mathcal{C}^\infty_c (\R \times (0,1))$ which is dense in both the $L^1$ and the $L^2$ norm induced by $V$, \emph{i.e.}, for any $g\in\mathcal{C}^\infty_c (\R \times (0,1))$, there are $g^{j_n}$, $n\ge1$, such that
\[\lim_{n\to\infty} \iint_{\Sigma_T} \big(g-g^{j_n}\big)^2V\,dxdt + \lim_{n\to\infty} \iint_{\Sigma_T} \big|g-g^{j_n}\big|V\,dxdt = 0.\]
For $\ell\ge1$, consider the functional $\Phi_\ell:\mathcal Y\to\R$ defined as
\[\Phi_\ell(\nu) := \max_{1 \le j \le \ell} \left\{ \iint_{\Sigma_T} g^jV \left[ \int_\R (\lambda-\rho)\nu_{t,x} (d\lambda) \right] dxdt - C_0\iint_{\Sigma_T} (g^j)^2V\,dxdt \right\}.\]
Note that $\Phi_\ell$ is continuous for each fixed $\ell$.
Lemma \ref{lemma:energy bound_countable set} below together with the weak convergence of $\Q_N$ then  shows that there is a constant $C_1$ independent of $\ell$, such that
\[E^{\Q} \big[\Phi_\ell\big] = \lim_{N\to\infty} E^{\Q_N} \big[\Phi_\ell\big] \le C_1, \quad \forall\,\ell\ge1.\]
Taking $\ell\to\infty$ and applying the monotone convergence theorem,
\[E^{\Q} \left[ \sup_{j\ge1} \left\{\iint_{\Sigma_T} (\bar u-\rho)g^jV\,dxdt - C_0\iint_{\Sigma_T} (g^j)^2V\,dxdt \right\} \right] \le C_1.
\]
The condition \eqref{eq:energy-bd-1} then follows from the dense property of $\{g^j;j\ge1\}$.

\begin{lemma}\label{lemma:energy bound_countable set}
There exist constants $C_0$ and $C_1$ such that, for each $\ell \ge 1$,
\begin{equation}\label{energy estimate_1}
	\limsup_{N \rightarrow \infty} \E_{\mu_N} \left[ \max_{1 \leq j \leq \ell} \left\{ \iint_{\Sigma_T} \! (u^N-\rho)g^jV\,dxdt - C_0\!\iint_{\Sigma_T} (g^j)^2V\,dxdt \right\} \right] \leq C_1.
\end{equation}
\end{lemma}

\begin{proof}
Only in this proof, to shorten the formulas we denote
\[\|g\|_V^2 := \iint_{\Sigma_T} g^2(t,x)V(x)\,dxdt, \quad \overline{(gV)}_i(t) := N\int_{\frac iN-\frac1{2N}}^{\frac iN+\frac1{2N}} g(t,x)V(x)\,dx.\]
Since $g^j$ is compactly supported, for $N$ sufficiently large,
\[\iint_{\Sigma_T} \big(u^N-\rho\big)g^jV\,dxdt = \int_0^T \frac1N\sum_{i=K}^{N-K} \hat\eta_i\overline{(g^jV)}_i\,dt - \iint_{\Sigma_T} \rho g^jV\,dxdt.\]
%To conclude the proof, it is sufficient to prove that there exist constants $C$ and $B$ such that
%\begin{equation}\label{energy_pf_1}
%	\limsup_{N \rightarrow \infty} \E_{\mu_N} \Big[\sup_{1 \leq j \leq \ell} \Big\{ \int_0^T \frac{1}{N} \sum_{i=K}^{N-K} (\hat{\eta}_i - \hat{\rho}_i) \overline{(Vg^j)}_i dt - C  \iint_{\Sigma_T} V(x) g^j(t,x)^2\,dx\,dt\Big\} \Big] \leq  B,
%\end{equation}
%and for each $1 \leq j \leq \ell$, 
%\begin{multline}\label{energy_pf_2}
%	\limsup_{N \rightarrow \infty}  \Big\{  \int_0^T \frac{1}{N} \sum_{i=K}^{N-K}  \hat{\rho}_i \overline{(Vg^j)}_i dt - \iint_{\Sigma_T} \rho(x)  V(x) g^j(t,x) dx\,dt \\- \iint_{\Sigma_T} V(x) g^j(t,x)^2\,dx\,dt\Big\}  \leq B.
%\end{multline}
To conclude the proof, it suffices to prove that
\begin{align}
	&\limsup_{N\to\infty} \left| \int_0^T \frac1N\sum_{i=K}^{N-K} \hat\rho_i\overline{(g^jV)}_i\,dt - \iint_{\Sigma_T} \rho g^jV\,dxdt \right| = 0, \quad \forall\,1 \le j \le \ell,\label{energy_pf_1}\\
	&\limsup_{N\to\infty} \E_{\mu_N} \left[ \max_{j \leq \ell} \left\{ \int_0^T \frac1N\sum_{i=K}^{N-K} (\hat\eta_i-\hat\rho_i)\overline{(g^jV)}_i\,dt - C_0\|g^j\|_V^2 \right\} \right] \le C_1.\label{energy_pf_2}
\end{align}

We begin with \eqref{energy_pf_1}, which is completely deterministic.
For some fixed $j\le\ell$, using the fact that $g^j$ is compactly supported, we only need to show that
\[\limsup_{N\to\infty} \left| \int_0^T dt\sum_{i=K}^{N-K} \int_{\frac iN-\frac1{2N}}^{\frac iN+\frac1{2N}} \big[\hat\rho_i-\rho(x)\big]g^j(t,x)V(x)\,dx \right| = 0.\]
This follows as $\rho$ is continuous and $\hat\rho_i$ is the smoothly weighted average of $\rho(\tfrac iN)$.

Now we prove \eqref{energy_pf_2}. First note that it suffices to prove it with $K$ replaced by any other mesoscopic scale $n=n(N)$ such that $K \le n = o(N)$, since
\[\limsup_{N\to\infty} \max_{j\le\ell} \left\{ \int_0^T \frac1N \Big(\sum_{i=K}^{n-1} + \sum_{i=N-n+1}^{N-K}\Big) \big|\overline{(g^jV)}_i\big| \right\} = 0.\]
%Therefore, it suffices to prove \eqref{energy_pf_2} with $K$ replaced by $n = n(N)$ that
The choice of $n$ is specified below in Lemma \ref{lem:V}.
Recall that $\nu_{\rho(\cdot)}^N$ is the product measure on $\Omega_N$ associated with the profile $\rho$.
From the entropy inequality, the expectation in \eqref{energy_pf_2} is bounded by
\[\frac{H(\mu_N \,|\, \nu_{\rho(\cdot)}^N)}{N} + \frac{1}{N} \log \E_{\nu_{\rho(\cdot)}^N} \left[\exp \left\{ \max_{j \leq \ell} \Big\{ \int_0^T \sum_{i=n}^{N-n} (\hat\eta_i-\hat\rho_i)\overline{(g^jV)}_i\,dt - C_0N\|g^j\|_V^2 \Big\} \right\} \right].\]
Due to \eqref{uniform_bound}, the relative entropy is bounded by $CN$, so the first term is uniformly bounded.
Also notice that for any random variables $X_1$, ... $X_\ell$,
\[\log E \big[e^{\max\{X_j;1 \leq j \leq \ell\}}\big] \leq \log E \Big[\sum_{1 \le j \le \ell} e^{X_j}\Big] \leq \log \ell + \max_{1 \leq j \leq \ell} \log E\big[e^{X_j}\big],\]
so we only need to find universal constants $C_0$ and $C_1$, such that
\begin{equation}\label{energy_pf_3}
	\log \E_{\nu_{\rho(\cdot)}^N} \left[\exp \left\{ \int_0^T \sum_{i=n}^{N-n} (\hat\eta_i-\hat\rho_i)\overline{(gV)}_i\,dt \right\} \right] \le \big(C_0\|g\|_V^2+C_1\big)N,
\end{equation}
for all $g\in\mathcal C_c^\infty(\R\times(0,1))$.
By the Feynman--Kac formula (see, e.g., \cite[Lemma 7.3]{baldasso2017exclusion}), the left-hand side of \eqref{energy_pf_3} is bounded from above by
\begin{equation}\label{energy_pf_4}
	\int_0^T \sup_{f} \left\{  \sum_{\eta \in \Omega_N} \sum_{i=n}^{N-n} \big(\hat{\eta}_i - \hat{\rho}_i\big) \overline{(gV)}_i f(\eta) \nu^N_{\rho(\cdot)} (\eta) + \big\langle \sqrt{f}, L_N\sqrt{f} \big\rangle_{\nu_{\rho(\cdot)}^N} \right\} dt,
\end{equation}
where the supremum is taken over all $\nu_{\rho(\cdot)}^N$-densities.
%Making the change of variables $\eta \mapsto \eta^j$, as $(1- \eta_j - \rho_j) \nu^N_{\rho(\cdot)} (\eta^j) = - (\eta_j - \rho_j) \nu^N_{\rho(\cdot)} (\eta)$, we have for each $j$ that
For each $j$, performing the change of variables $\eta\mapsto\eta^j$,
\begin{align*}
	&\sum_{\eta\in\Omega_N} (\eta_j - \rho_j)  f(\eta) \nu^N_{\rho(\cdot)} (\eta)\\
	 =\,&\frac12\sum_{\eta\in\Omega_N} (\eta_j - \rho_j)  f(\eta) \nu^N_{\rho(\cdot)} (\eta) + \frac12\sum_{\eta\in\Omega_N} (1- \eta_j - \rho_j)  f(\eta^j)  \nu^N_{\rho(\cdot)} (\eta^j) \\
	=\,&\frac{1}{2}\sum_{\eta\in\Omega_N} (\eta_j - \rho_j) \big[f(\eta)-f(\eta^j)\big] \nu^N_{\rho(\cdot)} (\eta),
\end{align*}
where the last line follows from the equality $(1- \eta_j - \rho_j) \nu^N_{\rho(\cdot)} (\eta^j) = - (\eta_j - \rho_j) \nu^N_{\rho(\cdot)} (\eta)$.
 Thus, the first term inside the supremum in \eqref{energy_pf_4} reads
 \begin{align*}
 	&\sum_{i=n}^{N-n} \overline{(gV)}_i \sum_{|j|<K} w_j \sum_{\eta\in\Omega_N} (\eta_{i+j}-\rho_{i+j}) f(\eta) \nu_{\rho(\cdot)}^N(\eta)\\
	=\,&\frac12\sum_{i=n}^{N-n} \overline{(gV)}_i \sum_{|j|<K} w_j \sum_{\eta\in\Omega_N} (\eta_{i+j}-\rho_{i+j}) \big[f(\eta)-f(\eta^{i+j})\big] \nu^N_{\rho(\cdot)}(\eta).
 \end{align*}
Using \CS inequality, we can bound it from above by $\mathcal I_1+\mathcal I_2$, where
\begin{align*}
	&\mathcal I_1 := \frac12 \sum_{i=n}^{N-n} \sum_{|j|<K} \sum_{\eta \in \Omega_N} w_j c_{i+j,{\rm G}}(\eta) \Big(\sqrt{f(\eta)} - \sqrt{f(\eta^{i+j})}\Big)^2 \nu^N_{\rho(\cdot)}(\eta),\\
	&\mathcal I_2 := \frac12 \sum_{i=n}^{N-n} \sum_{|j|<K} \sum_{\eta \in \Omega_N} w_j \frac{(\eta_{i+j} - \rho_{i+j})^2}{c_{i+j,{\rm G}}(\eta)} \overline{(gV)}_i^2 \Big(\sqrt{f(\eta)} + \sqrt{f(\eta^{i+j})}\Big)^2 \nu^N_{\rho(\cdot)}(\eta),
\end{align*}
where $c_{i,{\rm G}}(\eta)=V_i[\rho_i(1-\eta_i)+\eta_i(1-\rho_i)]>0$ due to \eqref{uniform_bound}.
Recall the Dirichlet form $D^N_{\rm G}$ defined in \eqref{eq:dirichlet_g}.
Since $\sum_{|j|<K} w_j=1$,
\[\mathcal I_1 \le \frac12\sum_{i=1}^{N-1} \sum_{|j|<K} w_j \sum_{\eta \in \Omega_N} c_{i,{\rm G}}(\eta) \Big(\sqrt{f(\eta)} - \sqrt{f(\eta^i)}\Big)^2 \nu^N_{\rho(\cdot)}(\eta) = N D^N_{\rm G} (\mu;\nu_{\rho(\cdot)}^N),\]
where $\mu:=f\nu_{\rho(\cdot)}^N$.
To estimate $\mathcal I_2$, notice that
\[\frac{(\eta_{i+j} - \rho_{i+j})^2}{c_{i+j,{\rm G}}(\eta)}\overline{(gV)}_i^2 \le \frac{C}{V_{i+j}}\int_{\frac iN-\frac1{2N}}^{\frac iN+\frac1{2N}} V\,dx \cdot \bigg[\int_{\frac iN-\frac1{2N}}^{\frac iN+\frac1{2N}} V\,dx\bigg]^{-1}\overline{(gV)}_i^2.\]
From \eqref{uniform_bound}, Lemma \ref{lem:V} below and \CS inequality,
\[\frac{(\eta_{i+j} - \rho_{i+j})^2}{c_{i+j,{\rm G}}(\eta)}\overline{(gV)}_i^2 \le CN\int_{\frac{i}{N}-\frac{1}{2N}}^{\frac{i}{N} + \frac{1}{2N}} V(x) g(t,x)^2 dx,\]
with some constant $C$ independent of $(i,j,g)$.
Therefore, $\mathcal I_2 \le C_0N\|g\|_V^2$.
Putting the estimates for $\mathcal I_1$, $\mathcal I_2$ into \eqref{energy_pf_4}, we see that \eqref{energy_pf_3} holds if we can show
\begin{equation}\label{es_1}
	ND^N_{\rm G} (\mu;\nu_{\rho(\cdot)}^N) + \langle \sqrt{f}, L_N \sqrt{f}\rangle_{\nu_{\rho(\cdot)}^N} \leq CN,
\end{equation}
for any $N$, any $\nu_{\rho(\cdot)}^N$-density $f$ and $\mu=f\nu_{\rho(\cdot)}^N$

The proof of \eqref{es_1} is standard.
Since $\nu_{\rho(\cdot)}^N$ is reversible for $L_{\rm G}$,
	\[D^N_{\rm G} (\mu;\nu_{\rho(\cdot)}^N) + \<\sqrt{f}, L_{\rm G} \sqrt{f}\>_{\nu_{\rho(\cdot)}^N} = 0.\]
In Appendix \ref{sec:df_estimates}, we prove that (cf.~\eqref{df_1} and \eqref{df_3})
\begin{align}
	\label{df_4}
	&\big\langle \sqrt f,L_{\rm exc}\sqrt f \big\rangle_{\nu_{\rho(\cdot)}^N} \le -\tfrac14(\sigma_N-1)D_{\rm exc}^N(\mu;\nu_{\rho(\cdot)}^N) + C,\\
	\label{df_5}
	&\big\langle \sqrt f,L_{\rm bd}\sqrt f \big\rangle_{\nu_{\rho(\cdot)}^N} \le - D_-^N(\mu;\nu_{\rho(\cdot)}^N) - D_+^N(\mu;\nu_{\rho(\cdot)}^N) + C.
\end{align}
Since the Dirichlet forms are non-negative, \eqref{es_1} follows.
\end{proof}

\begin{lemma}\label{lem:V}
Suppose that $V\in\mathcal C^1((0,1))$ and $\inf_{(0,1)} V>0$.
Define $n=n(N)$ as
\[n := \inf \left\{ n \ge K;\,\sup \Big\{|V'(x)|; x\in \left[ \tfrac{n-K}N,1-\tfrac{n-K}N \right]\Big\} \le \frac NK \right\}.\]
Then, $n=o(N)$ as $N\to\infty$, and there is a constant $C=C(V)$ such that
\[\max_{n \le i \le N-n,|j|<K} \left\{ V \left( \tfrac{i+j}{N} \right)^{-1} \int_{\frac iN-\frac1{2N}}^{\frac iN+\frac1{2N}} V\,dx \right\} \le \frac CN.\]
\end{lemma}

\begin{proof}
We first prove that for any $\varepsilon>0$, $n\le\varepsilon N$ for sufficiently large $N$.
Indeed, let $N_\varepsilon$ be such that $KN^{-1}<2^{-1}\varepsilon$ for all $N>N_\varepsilon$.
Then, if $N>N_\varepsilon$,
\[\sup \Big\{|V'(x)|;x\in \left[ \tfrac{\varepsilon N-K}N,1-\tfrac{\varepsilon-K}N \right] \Big\} \le \sup \Big\{|V'(x)|;x\in \left[ \tfrac\varepsilon2,1-\tfrac\varepsilon2 \right] \Big\} = C_\varepsilon.\]
We can further choose $N_\varepsilon$ such that $KN^{-1}<C_\varepsilon^{-1}$ for all $N>N_\varepsilon$, then $n\le\varepsilon N$.

For the second criterion, suppose that $|V(\tfrac{i+j}N)|$ takes the minimum value for $|j|<K$ at some $j_{N,i}$.
Then, for each $i=n$, ..., $N-n$,
\begin{align*}
	\left| N\int_{\frac iN-\frac1{2N}}^{\frac iN+\frac1{2N}} V\,dx - V \left( \tfrac{i+j_{N,i}}N \right) \right| &\le \sup \Big\{|V'|\mathbf1_{[\frac{i-K}N,\frac{i+K}N]} \Big\}\frac{|j_{N,i}|+1}N\\
	&\le \sup \Big\{|V'|\mathbf1_{[\frac{n-K}N,1-\frac{n-K}N]} \Big\}\frac KN \le 1.
\end{align*}
Therefore, for each $i$ and $j$,
\[\big[V \left( \tfrac{i+j}{N} \right) \big]^{-1}\int_{\frac iN-\frac1{2N}}^{\frac iN+\frac1{2N}} V\,dx \le \left[ V \left( \tfrac{i+j_{N,i}}{N} \right) \right]^{-1} \left[ 1+V \left( \tfrac{i+j_{N,i}}{N} \right) \right] \le 1+\frac1{\inf_{(0,1)} V}.\]
The second criterion then follows from \eqref{uniform_bound}.
\end{proof}

\subsection{Direct proof of Proposition \ref{prop:boundary_continuity}}
\label{subsec:micro-bd-cont}

%By Theorem \ref{thm:hydrodynamics}, for any $\varepsilon > 0$,
%	\[\lim_{N \rightarrow \infty} \int_0^t \frac{1}{\varepsilon N} \sum_{i=1}^{\varepsilon N} \eta_{i} (s) ds =  \frac{1}{\varepsilon} \int_{0}^{t} \int_{0}^{\varepsilon} u (s,x) dx ds \quad \text{in $\P_{\mu_N}$-probability}. \]
	By the continuity of $\rho(\cdot)$,
	\[\lim_{\varepsilon \rightarrow 0+} \lim_{N \rightarrow \infty} \frac{1}{\varepsilon N} \sum_{i=1}^{\varepsilon N} \rho_i =  \rho (0).\]
	Thus, we only need to show for fixed $\varepsilon>0$ that 
	\begin{equation}
	\lim_{N \rightarrow \infty} \int_0^t \frac{1}{\varepsilon N} \sum_{i=1}^{\varepsilon N} \big(\eta_{i} (s) - \rho_i\big) ds = 0 \quad \text{in $\P_{\mu_N}$-probability}.
	\end{equation}
Take 
\[A :=  A(\varepsilon) := \sqrt{\varepsilon / \inf \big\{ V(x);x\in(0,\varepsilon) \big\} }.\]
Let $\mu^{(\varepsilon N)} (s)$ denote the distribution of $\{\eta_1 (s), \eta_2 (s), \ldots, \eta_{\varepsilon N} (s)\}$.  By entropy inequality, 
\begin{align*}
	&\E_{\mu_N} \Big[ \Big| \int_0^t \frac{1}{\varepsilon N} \sum_{i=1}^{\varepsilon N} \big(\eta_{i} (s) - \rho_i\big) ds \Big|\Big] \leq \int_0^t  E^{\mu^{(\varepsilon N)} (s)} \Big[ \Big| \frac{1}{\varepsilon N} \sum_{i=1}^{\varepsilon N} \big(\eta_{i}  - \rho_i\big) \Big|\Big] ds \\
	&\leq \int_0^t \frac{H(\mu^{(\varepsilon N)} (s) | \nu^N_{\rho(\cdot)})}{AN} ds + \frac{t}{AN} \log E^{\nu^N_{\rho(\cdot)}}  \Big[ \exp \Big\{  \Big| \frac{A}{\varepsilon} \sum_{i=1}^{\varepsilon N} \big(\eta_{i}  - \rho_i\big) \Big| \Big\}\Big] =: {\rm I} + {\rm II}.
\end{align*}

We first bound ${\rm II}$, which is simpler. Note that we could first remove the absolute value inside the exponential. Since $\nu^N_{\rho(\cdot)}$ is a product measure, and by Taylor's expansion, there exists some constant $C$ such that
\begin{align*}
{\rm II} \leq \frac{t}{AN} \sum_{i=1}^{\varepsilon N} \log E^{\nu^N_{\rho(\cdot)}}  \Big[ \exp \Big\{ \frac{A}{\varepsilon}  \big(\eta_{i}  - \rho_i\big) \Big\}\Big] \leq \frac{CtA}{\varepsilon},
\end{align*}
which converges to zero as $\varepsilon \rightarrow 0$ by \eqref{v_condition_3}.

For ${\rm I}$, we consider the following Markov chain $X(t) := \{X_i(t)\}_{1\leq i \leq \varepsilon N}$, where $\{X_i (t)\}$, $1 \leq i \leq \varepsilon N$,  are independent $\{0,1\}$-valued Markov chains, and  the transition rates for $X_i (t)$ are given by 
\[1 \rightarrow 0 \quad \text{rate} \quad  V_i (1-\rho_i), \quad 0 \rightarrow 1 \quad \text{rate} \quad V_i \rho_i.\]
Since $\rho(\cdot)$ is bounded away from zero and one, the logarithmic Sobolev constant for the chain $X_i (t)$ is of order $V_i$. By \cite[Lemma 3.2]{diaconis1996logarithmic},  the logarithmic Sobolev constant for the chain $X(t)$ has order
\[\min_{1 \leq i \leq \varepsilon N} V_i \geq \inf \{V(x) ; x \in (0,\varepsilon)\}.\]
Therefore,
\[H(\mu^{(\varepsilon N)} (s) | \nu^N_{\rho(\cdot)}) \leq \frac{CN}{\inf \{V(x) ; x \in (0,\varepsilon)\}} D^N_{\rm G} (s).\]
By Lemma \ref{lem:dirichlet_form_estimate}, $\int_0^t D^N_{\rm G} (s) ds \leq C$. Thus,
\[{\rm I} \leq \frac{C}{A \inf \{V(x) ; x \in (0,\varepsilon)\}},\]
which also converges to zero as $\varepsilon \rightarrow 0$.

\appendix

\section{Computations concerning the Dirichlet forms}\label{sec:df_estimates}

Here we collect some fundamental estimates of the Dirichlet forms in \eqref{eq:dirichlet_exc}--\eqref{eq:dirichlet_right}.

\subsection{Proof of \eqref{df_1}}\label{sec:a1}
From the definition of $L_{\rm exc}$,
\begin{equation}\label{eq:ssep-asep}
	L_{\rm exc} = \frac{\sigma_N+1}2 S_{\rm exc} + \frac{2p-1}2 A_{\rm exc},
\end{equation}
where the operators $S_{\rm exc}$ and $A_{\rm exc}$ are respectively given by
\[S_{\rm exc}g = \sum_{i=0}^{N-1} \big[g(\eta^{i,i+1}) - g(\eta)\big], \quad A_{\rm exc}g = \sum_{i=0}^{N-1} (\eta_i-\eta_{i+1}) \big[g(\eta^{i,i+1}) - g(\eta)\big].\]
Recall that $\nu=\nu_{\frac12}^N$ and $f$ is a $\nu$-density.
By the basic inequality $x\log(y/x) \le 2\sqrt x(\sqrt y-\sqrt x)$ for any $x$, $y\ge0$,
\[\big\langle f,L_{\rm exc}[\log f] \big\rangle_\nu \leq 2 \big\langle \sqrt f,L_{\rm exc} [\sqrt f] \big\rangle_\nu.\]
We compute $S_{\rm exc}$ and $A_{\rm exc}$ respectively.
Since $\nu$ is reversible for $S_{\rm exc}$,
\[\big\langle \sqrt f,S_{\rm exc} \sqrt f \big\rangle_\nu = -D_{\rm exc}^N(\mu;\nu),\]
where $\mu=f\nu$.
For $A_{\rm exc}$, we only need to observe that
\[\big\langle \sqrt f,A_{\rm exc} \sqrt f \big\rangle_\nu = \sum_{\eta\in\Omega_N} (\eta_N-\eta_0)f(\eta)\nu(\eta) = E^\mu[\eta_N-\eta_0].\]
The estimate is then concluded.

\subsection{Proof of \eqref{df_2}}
Let $f_*$ be the density of $\mu$ with respect to $\nu_{\rho(\cdot)}^N$, then
\begin{align*}
	\log f(\eta^i) - \log f(\eta) &= \log f_*(\eta^i) - \log f_*(\eta) + \log \big[\nu_{\rho(\cdot)}^N(\eta^i)\big] - \log \big[\nu_{\rho(\cdot)}^N(\eta)\big]\\
	&= \log f_*(\eta^i) - \log f_*(\eta) + (1-2\eta_i) \log \left( \tfrac{\rho_i}{1-\rho_i} \right).
\end{align*}
Therefore, $\langle f,L_{\rm G}[\log f] \rangle_\nu$ equals to
\[\big\langle f_*,L_{\rm G}[\log f_*] \big\rangle_{\nu_{\rho(\cdot)}^N} + \sum_{\eta\in\Omega_N} \frac1N\sum_{i=1}^{N-1} f(\eta) c_{i,{\rm G}}(\eta) (1-2\eta_i) \log \left( \tfrac{\rho_i}{1-\rho_i} \right) \nu(\eta).\]
Since $\nu_{\rho(\cdot)}^N$ is reversible with respect to $L_{\rm G}$, similarly as in Appendix \ref{sec:a1},
\[\big\langle f_*,L_{\rm G}[\log f_*] \big\rangle_{\nu_{\rho(\cdot)}^N} \le 2 \big\langle \sqrt{f_*},L_{\rm G}\sqrt{f_*}\big\rangle_{\nu_{\rho(\cdot)}^N} = -2D_{\rm G}^N(\mu;\nu_{\rho(\cdot)}^N).\]
The estimate \eqref{df_2} then holds since $c_{i,{\rm G}}(\eta)(1-2\eta_i) = V_i(\rho_i-\eta_i)$.

\subsection{Proof of \eqref{df_3}}
As in Appendix \ref{sec:a1}, we begin with
\[\big\langle f,L_{\rm bd}[\log f] \big\rangle_\nu \le 2 \big\langle \sqrt f,L_{\rm bd}\sqrt f \big\rangle_\nu.\]
We compute the terms associated with $\eta_0$ as an example.
Those associated with $\eta_N$ follow the same argument.
Notice that
\begin{align*}
	&2\sum_{\eta\in\Omega_N} c_0(\eta)\sqrt{f(\eta)} \left[ \sqrt{f(\eta^0)} - \sqrt{f(\eta)} \right] \nu(\eta)\\
	=\,&-\sum_{\eta\in\Omega_N} c_0(\eta) \left[ \sqrt{f(\eta^0)}-\sqrt{f(\eta)} \right]^2 \nu(\eta) + \sum_{\eta\in\Omega_N} c_0(\eta) \big[f(\eta^0)-f(\eta)\big] \nu(\eta)\\
	=\,&-2D_-^N(\mu;\nu) + \sum_{\eta\in\Omega_N} f(\eta) \left[ c_0(\eta^0)\frac{\nu(\eta^0)}{\nu(\eta)}-c_0(\eta) \right] \nu(\eta).
\end{align*}
The conclusion then follows since $\nu(\eta^0)=\nu(\eta)$ and $c_0(\eta^0)-c_0(\eta) = (c_{\rm out}^--c_{\rm in}^-)(1-2\eta_0)$.

\subsection{Proof of \eqref{df_4}}
Let $f$ be a $\nu_{\rho(\cdot)}^N$-density function and denote $\mu=f\nu_{\rho(\cdot)}^N$.
Observe that the main difference from \eqref{df_1} is that the reference measure $\nu_{\rho(\cdot)}^N$ is not reversible.
Recall the decomposition \eqref{eq:ssep-asep} and we estimate $S_{\rm exc}$ and $A_{\rm exc}$ respectively.

We begin with the symmetric part $S_{\rm exc}$.
Let $g:=\sqrt f$.
By dividing $g(\eta)$ into $2^{-1}(g(\eta)-g(\eta^{i,i+1}))$ and $2^{-1}(g(\eta)+g(\eta^{i,i+1}))$, we obtain
\[\langle g,S_{\rm exc}g \rangle_{\nu_{\rho(\cdot)}^N} = -D_{\rm exc}^N(\mu;\nu_{\rho(\cdot)}^N) + \sum_\eta \sum_{i=0}^{N-1} \frac{g(\eta)+g(\eta^{i,i+1})}2 \big[g(\eta^{i,i+1})-g(\eta)\big] \nu_{\rho(\cdot)}^N(\eta).\]
Applying the change of variable $\eta^{i,i+1}\to\eta$, the second term becomes
\begin{align*}
	&\frac12\sum_\eta \sum_{i=0}^{N-1} \frac{g(\eta)[g(\eta^{i,i+1})-g(\eta)]}2 \left[ \nu_{\rho(\cdot)}^N(\eta)-\nu_{\rho(\cdot)}^N(\eta^{i,i+1}) \right]\\
	\le\,&\frac12D_{\rm exc}^N(\mu;\nu_{\rho(\cdot)}^N) + \frac14\sum_{\eta\in\Omega_N} \sum_{i=0}^{N-1} g^2(\eta) \bigg[1-\frac{\nu_{\rho(\cdot)}^N(\eta^{i,i+1})}{\nu_{\rho(\cdot)}^N(\eta)}\bigg]^2 \nu_{\rho(\cdot)}^N(\eta).
\end{align*}
Since $\rho \in\mathcal C^1([0,1];(0,1))$, with a constant $C$ independent of $(i,N)$ we have
\begin{equation}\label{exchange}
	\left| 1-\frac{\nu_{\rho(\cdot)}^N(\eta^{i,i+1})}{\nu_{\rho(\cdot)}^N(\eta)} \right| = \frac{|(\rho_{i+1}-\rho_i)(\eta_{i+1}-\eta_i)|}{\rho_i^{\eta_i}(1-\rho_i)^{1-\eta_i}\rho_{i+1}^{\eta_{i+1}}(1-\rho_{i+1})^{1-\eta_{i+1}}} \le \frac CN.
\end{equation}
Since $g^2\nu_{\rho(\cdot)}^N=\mu$ is a probability measure, we have
\[\langle g,S_{\rm exc}g \rangle_{\nu_{\rho(\cdot)}^N} \le -\frac12D_{\rm exc}^N(\mu;\nu_{\rho(\cdot)}^N) + \frac{C^2}{4N}.\]
Now we treat the antisymmetric part $A_{\rm exc}$.
First observe that
\begin{align*}
	\langle g,A_{\rm exc}g \rangle_{\nu_{\rho(\cdot)}^N} =\,&-\frac12\sum_\eta \sum_{i=0}^{N-1} (\eta_i-\eta_{i+1}) \big[g(\eta^{i,i+1})-g(\eta)\big]^2\nu_{\rho(\cdot)}^N(\eta)\\
 	&+ \frac12\sum_\eta \sum_{i=0}^{N-1} (\eta_i-\eta_{i+1})\big[g^2(\eta^{i,i+1})-g^2(\eta)\big]\nu_{\rho(\cdot)}^N(\eta).
 \end{align*}
As $|\eta_{i+1}-\eta_i|\le1$, the first term is bounded by $D^N_{\rm exc}(\mu;\nu_{\rho(\cdot)}^N)$.
The second term reads
\[\frac12\sum_\eta \sum_{i=0}^{N-1} (\eta_{i+1}-\eta_i)f(\eta) \nu_{\rho(\cdot)}^N(\eta^{i,i+1}) + \frac12\sum_\eta \sum_{i=0}^{N-1} (\eta_{i+1}-\eta_i) f(\eta)\nu_{\rho(\cdot)}^N(\eta).\]
Since $f(\eta)\nu_{\rho(\cdot)}^N(\eta)=\mu(\eta)$, its modulus is bounded by
\begin{align*}
	\frac12\sum_\eta \sum_{i=0}^{N-1} |\eta_{i+1}-\eta_i|\left| \frac{\nu_{\rho(\cdot)}^N(\eta^{i,i+1})}{\nu_{\rho(\cdot)}^N(\eta)} - 1 \right| \mu(\eta) + E^\mu [\eta_N-\eta_0],
\end{align*}
which is uniformly bounded due to \eqref{exchange}.
Therefore,
\[\langle g,A_{\rm exc}g \rangle_{\nu_{\rho(\cdot)}^N} \le D_{\rm exc}^N(\mu;\nu_{\rho(\cdot)}^N) + C.\]
Putting the two estimates into \eqref{eq:ssep-asep}, we can conclude since $\sigma_N \ll N$.

\subsection{Proof of \eqref{df_5}}
Repeating the argument in the proof of \eqref{df_4} with $\nu$ replaced by $\nu_{\rho(\cdot)}^N$, we only need to bound
\[\frac12\sum_{\eta\in\Omega_N} f(\eta) \left[ c_0(\eta^0)\frac{\nu_{\rho(\cdot)}^N(\eta^0)}{\nu_{\rho(\cdot)}^N(\eta)}-c_0(\eta) \right] \nu_{\rho(\cdot)}^N(\eta).\]
It is uniformly bounded since
\[ \left| c_0(\eta^0)\frac{\nu_{\rho(\cdot)}^N(\eta^0)}{\nu_{\rho(\cdot)}^N(\eta)}-c_0(\eta) \right| = \left| c_0(\eta^0)\Big(\frac{\rho_0}{1-\rho_0}\Big)^{1-2\eta_0} - c_0(\eta) \right| \le C.\]

\section*{Acknowledgement}
The authors specially thank an anonymous referee whose earnest work helped to improve the paper.
Linjie Zhao is supported by National Natural Science Foundation of China (Grant Number 12371142) and the financial support from the Fundamental Research Funds for the Central Universities in China.

\section*{Statements}

\subsection*{Conflict of interest}
All authors have no conflict of interest.

\subsection*{Data availability}
The authors declare that all data supporting this article are available within the article.

%\printbibliography

\end{document}